\documentclass[11pt, a4paper, reqno]{amsart}

\usepackage{amscd}
\usepackage{dsfont}
\usepackage{amsmath, upgreek}
\usepackage{amstext}
\usepackage{amsthm}
\usepackage{amssymb}
\usepackage{amsopn}
\usepackage{amssymb}
\usepackage{amsxtra}
\usepackage{amsfonts}
\usepackage{fancyhdr}
\usepackage[mathcal]{euscript}
\usepackage[english]{babel}
\usepackage[latin1]{inputenc}
\usepackage{bbold}
\usepackage{setspace}
\usepackage[svgnames]{xcolor}
\usepackage{adjustbox}

\usepackage{tikz}
\usepackage{tikz-cd}
\usepackage{tkz-euclide}
\usepackage{euscript}
\usepackage{hyperref}
\hypersetup{
    colorlinks = true,
    linkcolor = RoyalBlue,
    citecolor = Green,    
}

\renewcommand{\emptyset}{\varnothing}
\makeatletter
\g@addto@macro\normalsize{%
  \setlength\abovedisplayskip{10pt}
  \setlength\belowdisplayskip{10pt}
  \setlength\abovedisplayshortskip{5pt}
  \setlength\belowdisplayshortskip{8pt}
}

\newtheorem{thm}{Theorem}[section]
\newtheorem{cor}[thm]{Corollary}
\newtheorem{lem}[thm]{Lemma}
\newtheorem{prop}[thm]{Proposition}

\theoremstyle{definition}
\newtheorem{ex}[thm]{Example}
\newtheorem{rmk}[thm]{Remark}
\newtheorem{dfn}[thm]{Definition}
\newtheorem{ass}[thm]{Assumption}
\newtheorem{question}[thm]{Question}

\newcommand{\gldim}{\operatorname{gl.\hspace{-1.5pt}dim}}
\newcommand{\resdim}{\operatorname{res.\hspace{-1.5pt}dim}}

\renewcommand{\epsilon}{\varepsilon}
\newcommand{\gr}{\operatorname{gr}\nolimits}
\newcommand{\Hom}{\operatorname{Hom}\nolimits}
\newcommand{\Ext}{\operatorname{Ext}\nolimits}
\newcommand{\mac}{\operatorname{mac}\nolimits}
\newcommand{\inc}{\operatorname{inc}\nolimits}
\newcommand{\cone}{\operatorname{cone}\nolimits}

\newcommand{\spann}{\operatorname{span}\nolimits}

\newcommand{\id}{\operatorname{id}\nolimits}
\newcommand{\Bigosum}[2]{\raisebox{0.5pt}{\ensuremath{\displaystyle\mathop{\textstyle\bigoplus}_{#1}^{#2}}}}
\newcommand{\Bigsum}[2]{\ensuremath{\mathop{\textstyle\sum}_{#1}^{#2}}}
\newcommand{\ran}{\operatorname{ran}\nolimits}
\newcommand{\cok}{\operatorname{cok}\nolimits}
\newcommand{\coim}{\operatorname{coim}\nolimits}
\newcommand{\im}{\operatorname{im}\nolimits}
\newcommand{\ind}{\operatorname{colim}\nolimits}
\newcommand{\w}{\operatorname{w}\nolimits}
\newcommand{\Proj}{\operatorname{Proj}\nolimits}

\newcommand{\Stovicek}{{\v{S}}{\v{t}}ov{\'{\i}}{\v{c}}ek}

\newcommand{\DD}{\ensuremath{\text{\bf\textsf{D}}}}
\newcommand{\K}{\ensuremath{\text{\bf\textsf{K}}}}
\newcommand{\C}{\ensuremath{\text{\bf\textsf{Ch}}}}
\newcommand{\Ac}{\ensuremath{\text{\bf\textsf{Ac}}}}
\newcommand{\bb}{\ensuremath{\text{\bf\textsf{b}}}}

\newcommand{\NN}{\mathbb{N}}

\newcommand{\CC}{\mathbb{C}}

\newcommand{\KK}{\mathbb{K}}
\newcommand{\EE}{\mathbb{E}}
\newcommand{\ZZ}{\mathbb{Z}}

\newcommand{\BB}{\mathbb{B}}

\newcommand{\diam}{\hfill\raisebox{-0.75pt}{\scalebox{1.4}{$\diamond$}}}

\renewcommand{\bullet}{\,\begin{picture}(-1,1)(-1,-3)\circle*{2.}\end{picture}\ }

\newcommand{\Class}{\mathsf{C}}

\newcommand{\HDLCS}{\mathsf{Tc}}
\newcommand{\Bor}{\mathsf{Born}}
\newcommand{\Ban}{\mathsf{Ban}}
\newcommand{\Fre}{\mathsf{Fre}}

\newcommand{\LB}{\mathsf{LB}}
\newcommand{\LBr}{\mathsf{LB}_{\mathsf{reg}}}
\newcommand{\LBc}{\mathsf{LB}_{\mathsf{com}}}

\newcommand{\Call}{\mathbb{C}_{\mathsf{all},\hspace{1pt}\LB}}
\newcommand{\Callreg}{\mathbb{C}_{\mathsf{all},\hspace{1pt}\LBr}}
\newcommand{\Callcom}{\mathbb{C}_{\mathsf{all},\hspace{1pt}\LBc}}

\newcommand{\Etop}{\mathbb{E}_{\mathsf{top},\hspace{1pt}\LB}}
\newcommand{\Emax}{\mathbb{E}_{\mathsf{max},\hspace{1pt}\LB}}

\newcommand{\Etopreg}{\mathbb{E}_{\mathsf{top},\hspace{1pt}\LBr}}
\newcommand{\Emaxreg}{\mathbb{E}_{\mathsf{max},\hspace{1pt}\LBr}}

\newcommand{\Etopcom}{\mathbb{E}_{\mathsf{top},\hspace{1pt}\LBc}}
\newcommand{\Emaxcom}{\mathbb{E}_{\mathsf{max},\hspace{1pt}\LBc}}

\newcommand{\Cregsur}{\mathbb{D}_{\mathsf{max},\hspace{1pt}\LBr}}

\newcommand{\Dcom}{\mathbb{D}_{\mathsf{max},\hspace{1pt}\LBc}}

\newcommand{\Dmax}{\mathbb{D}_{\mathsf{max},\hspace{1pt}\LB}}

\renewcommand{\sharp}{\texttt{\#}}

\newcommand{\hypref}[1]{\hyperref[R0]{\color{RoyalBlue}{\bf #1}}}

\usepackage{xstring}
\DeclareMathOperator*{\fsum}{%
  \mathchoice
    {\raisebox{-.0\height}{\scalebox{1.15}{$\sum$}}}
    {\raisebox{-.0\height}{\scalebox{1}{$\sum$}}}
    {\raisebox{-.0\height}{\scalebox{0.75}{$\sum$}}}
    {\raisebox{-.05\height}{\scalebox{0.55}{$\sum$}}}}

\newcommand\zfrac[2]{\text{\footnotesize\raisebox{.15ex}{%
\dimen0=\fontdimen8\textfont2  
\dimen2=\fontdimen11\textfont2 
\dimen4=\fontdimen8\textfont3  
$%
\fontdimen8\textfont2=.5\dimen0
\fontdimen11\textfont2=.5\dimen2
\fontdimen8\textfont3=1.1\dimen4
\dfrac{#1}{#2}$%
\fontdimen8\textfont2=\dimen0
\fontdimen11\textfont2=\dimen2
\fontdimen8\textfont3=\dimen4
}}}

\renewcommand*\textcircled[1]{\tikz[baseline=(char.base)]{
            \node[shape=circle,draw,inner sep=1pt] (char) {#1};}}

	\usetikzlibrary{matrix,arrows,decorations,decorations.pathmorphing,positioning,decorations.pathreplacing,shapes}
	\tikzset{commutative diagrams/.cd, 
		mysymbol/.style = {start anchor=center, end anchor = center, draw = none}}

\newcommand{\commutes}[2][\circ]{\arrow[mysymbol]{#2}[description]{#1}}

\DeclareMathSymbol{\sm}{\mathbin}{AMSa}{"39}
\DeclareMathSymbol{\shortminus}{\mathbin}{AMSa}{"39}

\DeclareMathSymbol{\shortminus}{\mathbin}{AMSa}{"39}

\usetikzlibrary{arrows.meta}

\usepackage{enumitem}

\newlist{compactitem}{itemize}{3} 
\setlist[compactitem,1]{leftmargin=26pt, noitemsep, topsep=0pt}

\newlist{myitemize}{itemize}{1}
\setlist[myitemize,1]{leftmargin=26pt, noitemsep, topsep=0pt}

\renewcommand{\tau}{\uptau}
\renewcommand{\sigma}{\upsigma}

\begin{document}

\allowdisplaybreaks
$ $
\vspace{-20pt}

\title{A homological approach to (Grothendieck's)\\[2pt]completeness problem for regular LB-spaces}



\author{Sven-Ake Wegner\hspace{0.5pt}\MakeLowercase{$^{\text{1},\,\text{2}}$}}

\renewcommand{\thefootnote}{}
\hspace{-1000pt}\footnote{\hspace{5.5pt}2020 \emph{MSC}: Primary 46M18, 46A13, 18G80; Secondary 18E05, 46M10, 46A45.\vspace{1pt}}


\hspace{-1000pt}\footnote{\hspace{5.5pt}\emph{Key words}: LB-space, inductive limit, exact category, karoubian category, derived category. \vspace{1pt}}

\hspace{-1000pt}\footnote{\hspace{0pt}$^{1}$\,University of Hamburg, Department of Mathematics, Bundesstra\ss{}e 55, 20146 Hamburg, Germany, phone:\newline\phantom{x}\hspace{1.2pt}+49\,(0)\,40\:2395\hspace{1pt}-\hspace{1pt}25120, e-mail: sven.wegner@uni-hamburg.de.\vspace{1.6pt}}

\hspace{-1000pt}\footnote{\hspace{0pt}$^{2}$\,This research was funded by the German Research Foundation (DFG); Project no.~507660524.}

\begin{abstract}
We consider the long-standing question of whether every regular LB-space is complete. This problem has been open since the 1950s and originates in Grothendieck's early work in functional analysis. Rather than seeking a direct proof or counterexample, our approach is to study weak versions of the problem using homological methods. We consider the categories of complete and, respectively, regular LB-spaces, establish that their derived categories are well-defined with respect to several exact structures, and show that there are canonical triangle functors between them. If one of these functors were not an equivalence, this would provide a negative answer to Grothendieck's question. In contrast, we prove that one of them is an equivalence. This may be interpreted as evidence in favor of an affirmative answer to the original problem, and it shows in particular that the two classes of spaces share the same homological algebra, even if they were to differ.
\end{abstract}

\maketitle

\vspace{0pt}
\section{Introduction}\label{SEC-1}

In this article, by an LB-space $(X,\uptau)=\ind_{n\in\NN}X_n$ we mean a Hausdorff locally convex topological vector space that can be written as an inductive limit of a sequence $X_0\hookrightarrow X_1\hookrightarrow\cdots$ of Banach spaces with linear and continuous inclusions as structure maps. As a vector space, $X$ is the union of the $X_n$ and $\uptau$ is the finest locally convex topology that makes all inclusions $X_n\hookrightarrow(X,\tau)$ continuous. The study of LB-spaces began around 75 years ago, when K\"othe \cite{K48} defined the coechelon space $k^{\hspace{1pt}1}(V)=\ind_{n\in\NN}\ell^{\hspace{1pt}1}(v_n)$ in 1948. Shortly thereafter, in 1949\hspace{1pt}-\hspace{1pt}50, Dieudonn\'e, Schwartz \cite{DS49} and K\"othe \cite{K50} coined the notion of `LF-spaces' for countable inductive limits of Fr\'echet spaces; this was picked up by Grothendieck \cite{G, G55} who contributed sig\-ni\-fi\-cant\-ly to the general theory of LF- and LB-spaces during the early 1950s. While Grothendieck did not use the term `LB-space' explicitly, the latter notation became standard in the 1960s, at the latest with the appearance of K\"othe's monograph \cite{KI}.

\smallskip

The reason spaces of the above types became popular is on the one hand that many classic spaces of analysis are indeed LF- or duals of LF-spaces. Prominent examples are  test functions $\operatorname{C}_{\operatorname{c}}^{\infty}(\Omega)$, distributions $\mathcal{D}'(\Omega)$, real analytic functions $\mathcal{A}(\Omega)$, compactly supported continuous functions $\operatorname{C}_{\operatorname{c}}(\Omega)$, or germs of holomorphic functions $\operatorname{H}(K)$. On the other hand, inductive limits are interesting as they in general behave rather badly: Firstly, the limit space $(X,\uptau)$ may be incomplete, even if all step spaces are Banach, see K\"othe \cite{K50}. Secondly, while a bounded subset $B\subseteq X_n$ is automatically bounded in $(X,\uptau)$, it might happen that not every bounded set in $(X,\uptau)$ arises in this way! The latter was pointed out by Grothendieck \cite[Chapitre I, p. 12\hspace{1.5pt}-13]{G55} in reference to the very example given in \cite{K50}. Despite this, Grothendieck's factorization theorem \cite[Chapitre I, p.~16]{G55} implies that every complete LF-space is `regular', meaning that for every $\uptau$-bounded $B\subseteq X$ there exists $n\in\NN$ such that $B$ is contained and bounded in $X_n$. In fact, the following hierarchy holds for LF- resp.~LB-spaces, see Bierstedt \cite[p.~77]{Klaus}:\vspace{0pt}
\begin{equation*}
\adjustbox{scale=0.925,center}{\begin{tikzcd}
\text{\small complete}\;\arrow[Rightarrow, shift left=0.65ex, xshift=0.25ex]{r}{}\arrow[Leftarrow, shift left=-0.65ex]{r}{\begin{picture}(0,0)\put(-10,-23.5){$\stackrel{\uparrow}{\stackrel{\text{known}}{\scriptscriptstyle\text{for LB}}}$}\end{picture}} &\;\text{\begin{minipage}{50pt}\centering\small quasi-\vspace{-3pt}\\complete\end{minipage}}\arrow[Rightarrow]{r}{} &\;\text{\begin{minipage}{50pt}\centering\small sequentially\vspace{-3pt}\\complete\end{minipage}}\;\arrow[Rightarrow]{r}{} &\text{\begin{minipage}{50pt}\centering\small Mackey\vspace{-3pt}\\complete\end{minipage}}\arrow[Rightarrow, shift left=0.65ex, xshift=0.25ex]{r}{}\arrow[Leftarrow, shift left=-0.65ex]{r}{} &\,\;\text{\small regular.}
\end{tikzcd}}\vspace{10pt}
\end{equation*}
The non-obvious implications are due to Grothendieck who asked in \cite[Chapitre II, p.~137, question non r\'esolue no.\ 9]{G55} explicitly if the first reverse implication holds in general for LF-spaces. Bierstedt \cite[Problem 1 on p.~78]{Klaus}, followed by many others, e.g., \cite[Problem 13.8.6]{BPC}, \cite[p.\ 150]{BDM94}, \cite[p.\ 247]{JochenAcyc}, also attributed to him the question of whether all properties above are equivalent:
\begin{question}(attributed to Grothendieck)\label{Q-G} Is every regular LF-space complete? 
\end{question}
There are large subclasses of LF-spaces that are automatically complete: This is the case for `strict LF-spaces' where by assumption all structure maps $X_n\hookrightarrow X_{n+1}$ are relatively open, or `LS-spaces' where the structure maps are assumed to be compact. For such classes Question \ref{Q-G} is moot. Beyond this there are only two general results known: The first is due to Wengenroth \cite[Thm 3.3]{JochenAcyc} who showed in 1996 that the implication `regular\,$\Longrightarrow$\,complete' holds for LF-spaces whose steps are Montel. The second is due to Bonet, Dierolf, Ku\ss{} \cite{BD89, DK} who showed the same for the class of so-called Moscatelli-type LB-spaces\,---\,which was in particular surprising as this class had served very well in the construction of counterexamples to other questions in the past. To conclude this historical overview of Question \ref{Q-G}, we have to warn the reader that between 1989 and 2002 several papers claimed to have answered it, first in the negative and then positively. Although all these articles have significant gaps, see Bonet \cite{Kucera}, they still appear at the top of search engine results and are quoted by chatbots.

\smallskip

Having said all of the above, our aim in this paper is not to attempt to answer Question \ref{Q-G} as stated, but rather to introduce and discuss a homological version of its LB-case by using derived categories\hspace{0.5pt}---\hspace{0.5pt}a concept that, curiously enough, also traces back to Grothendieck and his student Verdier, see \cite{Verdier}. 

\smallskip

As a matter of fact, homological methods have been used in functional analysis with great success already since the late 1960s. Beginning with Palamodov's seminal papers \cite{Pala68, Pala71}, we refer to Wengenroth's book \cite{JochensBuch} for a thorough exposition and additional references. From the 2000s onwards, derived categories were also considered, primarily in bornological analysis and geometry, e.g.\ \cite{PSDuke, Meyer, BBBK18,  BKK24, KellysBook}, but also in the context of topological groups or vector spaces, e.g.\ \cite{Prosmans, BuehlerCR, Pirk, HKRW, LW, Lupini, Positselski,  Braunling}. While homological algebra classically deals with abelian categories, none of the categories mentioned above enjoys this property. Consequently, the homological toolbox, and in particular the definition of the derived category, needs to be generalized in order to become applicable in the context of functional analysis. In order to do so, there are two different approaches: On the one hand, one can look at \emph{intrinsic} properties that are weaker than abelianness but still strong enough to ensure that suitable versions of the (diagram) lemmas on which homological methods rely can be proven. This leads to the notions of pre-, semi- and quasiabelian categories and corresponding homological theories, see for example \cite{BP65, Pala68, Raikov, Pala71, KC72, Yakovlev, Generalov, KK03} and the references therein. In particular for quasiabelian categories the derived category has been studied extensively in the monograph \cite{S} by Schneiders and this class covers among others the categories $\HDLCS$ of locally convex spaces, $\Ban$ of Banach spaces, $\Fre$ of Fr\'echet spaces and $\Bor$ of bornological spaces. It however does not cover the three categories
$$
\LBc\subseteq\LBr\subset\LB
$$
of complete/regular/all LB-spaces that we are interested in. The second approach to a non-abelian homological algebra is to endow a category with an \emph{extrinsic} notion of short exact sequences, or, in Quillen's \cite{Quillen72} notation, with an exact structure. Generalizing Be\u{\i}linson, Bernstein, Deligne \cite{BBD}, Neeman \cite{Nee90} showed that the derived category with respect to an arbitrary extrinsic exact structure can be defined, and that it behaves particularly well under the mild intrinsic condition of karoubianness. In Sections \ref{SEC-2}\hspace{0.75pt}--\hspace{0.5pt}\ref{SEC-COM} we establish that $\LBr$ and $\LBc$ are karoubian and that they can be endowed with two natural exact structures each, namely with the `maximal exact structure' or with the structure consisting of the `topologically exact sequences', which is the restriction of the intrinsic notion of exactness in $\HDLCS$.

\smallskip

For some categories arising in functional analysis, e.g., $\Ban$ or $\Fre$, a sequence $X\stackrel{\scriptscriptstyle f}{\rightarrow} Y\stackrel{\scriptscriptstyle g}{\rightarrow} Z$ with linear and continuous $f$ and $g$ belongs to the maximal exact structure iff it is topologically exact iff it is algebraically exact, i.e., exact if we forget the topologies and read it as a sequence of vector spaces only. This however does not hold for $\LB$, meaning in particular that the algebraically exact sequences do not constitute an exact structure on that category. As they are certainly the easiest class to work with, the question arises if one can nevertheless do homological algebra with them. It turned out very recently that this is indeed possible using `one-sided exact structures' as initially introduced by Rump \cite{Rump01}, Bazzoni, Crivei \cite{BC13} and further studied by Henrard, Kvamme, van Roosmalen and the author \cite{HR19, HR19a, HR20,HKRW, HR24} under the name of `deflation- and inflation-exact categories'. In Section \ref{SEC-2a} of this article we give a short summary of their theory and in Sections \ref{SEC-2}\hspace{0.75pt}--\hspace{0.5pt}\ref{SEC-COM} we show that $\LBr$ and $\LBc$ each admit a maximal deflation-exact structure. In \cite{HKRW} it was shown that the maximal deflation-exact structure of $\LB$ consists precisely of the algebraically exact sequences; here we show the same for $\LBr$, while for $\LBc$ the maximal deflation-exact structure is at least contained in the algebraically exact sequences.

\smallskip

Our characterizations of the (deflation-)exact structures on $\LBr$ and $\LBc$, plus the fact that both categories are karoubian, provide three different ways to define their derived categories. Using a consistent choice of exact structures on both sides yields exact inclusion functors $\LBc\hookrightarrow\LBr$, see the diagram in Section \ref{SEC-EX-FKT}. In Section \ref{SEC-AC} we describe the acyclic complexes in $\LB$, $\LBr$ and $\LBc$ in analytic terms. Combining the results of the previous sections, we obtain in Section \ref{SEC-7} that the aforementioned exact inclusions induce three natural triangle functors
\begin{eqnarray*}
\operatorname{F}_{\mathsf{top}}^{\text{b}}\colon&\hspace{-12pt}\DD^{\bb}(\LBc,\Etopcom)\longrightarrow\DD^{\bb}(\LBr,\Etopreg),\\[1pt]
\operatorname{F}_{\mathsf{max}}^{\text{b}}\colon&\hspace{-12pt}\DD^{\bb}(\LBc,\Emaxcom)\longrightarrow\DD^{\bb}(\LBr,\Emaxreg),\\[1pt]
\operatorname{F}_{\mathsf{def}}^{\text{b}}\colon&\hspace{-12pt}\DD^{\bb}(\LBc,\Dcom)\longrightarrow\DD^{\bb}(\LBr,\Cregsur).
\end{eqnarray*}
These then provide the following `homologifications' of Question \ref{Q-G}:
\begin{question}\label{Q-H} Are any of the functors $\operatorname{F}_{\mathsf{top}}^{\text{b}}$, $\operatorname{F}_{\mathsf{max}}^{\text{b}}$ or $\operatorname{F}_{\mathsf{def}}^{\text{b}}$ an equivalence?
\end{question}
While we show in this article that Question \ref{Q-H} is well-defined, we are unfortunately not able to answer it. We would however like to point out that if one of the functors were not an equivalence, this would yield a negative answer to the original Question \ref{Q-G}. Conversely, if some of the functors above were equivalences, this may be interpreted as evidence in favor of an affirmative answer to Grothendieck's original problem.

\smallskip

The obstruction for answering Question \ref{Q-H} in the affirmative is the following: Although it is well-known that any LB-space $X=\mathop{\ind}_{n\in\NN}X_n$ has the following `standard resolution' 
\begin{equation}\label{EQ-INTRO}\tag{0}
\Bigosum{n\in\mathbb{N}}{}X_n\stackrel{d}{\longrightarrow}\Bigosum{n\in\mathbb{N}}{}X_n\stackrel{\sigma}{\longrightarrow}X,
\end{equation}
in which the direct sums are complete LB-spaces, it is also well-known that in the above sequence the map $d$ is, even for regular $X$, in general not an isomorphism/weak isomorphism onto its range. The latter means that \eqref{EQ-INTRO} in general is not an $\Etopreg$/$\Emaxreg$-acyclic complex and thus not the type of resolution one needs in order to see that $\operatorname{F}_{\mathsf{top}}^{\text{b}}$/$\operatorname{F}_{\mathsf{max}}^{\text{b}}$ is essentially surjective. In contrast, our results yield that the standard resolution is indeed $\Cregsur$-acyclic. Moving forward, a calculus-of-fractions proof for $\operatorname{F}_{\mathsf{def}}^{\text{b}}$ being full or faithful founders on the fact that it is unclear whether a $\Cregsur$-acyclic complex, whose objects are all complete, is already $\Dcom$-acyclic. 

\smallskip

In order to address the above, in the second part of Section \ref{SEC-7} we generalize \Stovicek, Kerner, Trlifaj's \cite{SKT11} notion of `relative derived categories'. Their idea was essentially to define the derived category of a subcategory $\mathcal{B}\subseteq\mathcal{A}$ of an abelian category by making those morphisms of complexes invertible whose cone is a complex over $\mathcal{B}$ and $\mathcal{A}$-acyclic. Replacing `abelian' with `deflation-exact' and then using the same definition via a Verdier quotient, it turns out that by  setting $\mathcal{A}=(\LB,\Dmax)$ and choosing $\mathcal{B}\in\{\LBc,\LBr\}$ all technical issues explained in the paragraph above can be circumvented and the following homological version of Question \ref{Q-G} (see Theorem \ref{DE-COR}) can be proven:

\begin{thm}The inclusion functor $\LBc\hookrightarrow\LBr$ induces a triangle equivalence 
$$
\operatorname{F}_{\mathsf{rel}}^{\text{\emph{b}}}\colon\DD^{\bb}\bigl(\LBc;(\LB,\Dmax)\bigr)\stackrel{\hspace{-2pt}\sim}{\longrightarrow}\DD^{\bb}\bigl(\LBr;(\LB,\Dmax)\bigr)
$$
of relative derived categories.\qed
\end{thm}

We emphasize that the above means that even if the classes of regular and complete LB-spaces do not coincide, their relative homological algebra inside the class of all LB-spaces with respect to all algebraically exact sequences is nevertheless the same.

\section{Conflation categories}\label{SEC-2a}

Let $\mathcal{A}$ be an additive category. We call a morphism $f$ in $\mathcal{A}$ \emph{a kernel} if there is some morphism $g$ in $\mathcal{A}$ such that $f=\ker g$. Dually, $f$ is \emph{a cokernel} if there is some $g$ with $f=\cok g$. The category $\mathcal{A}$ is said to be \emph{preabelian} if every morphism has a kernel and a cokernel. A pair $(f,g)$ of morphisms is a \emph{kernel-cokernel pair} if $f=\ker g$ and $g=\cok f$ both hold.

\smallskip

The following notion goes back to Richman, Walker \cite{RW} and has been generalized to non-preabelian categories by Crivei \cite{Crivei}.

\begin{dfn}\label{AN-3}\cite[Dfn 2.4]{Crivei} Let $\mathcal{A}$ be an additive category.  We say that $f\colon X\rightarrow Y$ is a \emph{semistable kernel} if it is a kernel and for any morphism $t\colon X\rightarrow T$ in $\mathcal{A}$ the pushout\vspace{-10pt}
 \begin{equation*}
\begin{tikzcd}
X\arrow{r}{f}\arrow{d}[swap]{t}\commutes[\mathrm{PO}]{dr} & Y \arrow{d}{q_Y}\\[4pt]
T \arrow{r}[swap]{q_T} & Q.
\end{tikzcd}
\end{equation*}
of $f$ along $t$ exists and $q_T$ is again a kernel. \emph{Semistable cokernels} are defined dually. A kernel-cokernel pair $(f,g)$ is a \emph{stable kernel-cokernel pair} if $f$ is a semistable kernel and $g$ is a semistable cokernel.\diam{}
\end{dfn}

Throughout this article we consider exact categories in the sense of Quillen \cite{Quillen72} as well as their one-sided generalization due to Rump \cite{RumpMax}, Bazzoni, Crivei \cite{BC13} and Henrard et al.~\cite{HKR, HKRW, HR19, HR19a, HR20, HR24}. We adopt the notation of \cite{HKR, HKRW, HR20} but mention below how the latter relates to the other articles.

\begin{dfn}\label{DEFLEXCAT} \cite[Dfns 2.7 and 4.6]{HKRW} We call a pair $(\mathcal{A},\mathbb{C})$, consisting of an additive category $\mathcal{A}$ and a class $\mathbb{C}$ of kernel-cokernel pairs in $\mathcal{A}$, a \emph{conflation category}. If $(f,g)\in\CC$, then we refer to $f$ as an \emph{inflation} and to $g$ as a \emph{deflation}.\vspace{3pt}
\begin{compactitem}
\item[(i)] A conflation category $(\mathcal{A},\mathbb{C})$ is \emph{deflation-exact} if the following three axioms hold.

\vspace{3pt}

\begin{compactitem}
\item[\color{RoyalBlue}{\textbf{R0}}\label{R0}] For each $X\in\mathcal{A}$ the map $X\rightarrow0$ is a deflation.\vspace{3pt}
\item[\color{RoyalBlue}{\textbf{R1}}\label{R1}] The composition of two deflations is a deflation.\vspace{3pt}
\item[\color{RoyalBlue}{\textbf{R2}}\label{R2}] The pullback of a deflation along any morphism exists and is again a deflation.
\end{compactitem}

\vspace{3pt}

\item[(ii)] The conflation category $(\mathcal{A},\mathbb{C})$ is \emph{strongly deflation-exact} if additionally the following axiom holds.

\vspace{3pt}

\begin{compactitem}
\item[\color{RoyalBlue}{\textbf{R3}}\label{R3}] If $i$ and $p$ are composable morphisms in $\mathcal{A}$ such that $p$ has a kernel and $p\circ i$ is a deflation, then $p$ is a deflation.
\end{compactitem}

\vspace{3pt}

\item[(iii)] We say that a conflation category $(\mathcal{A},\mathbb{C})$ has \emph{admissible kernels} if in $\mathcal{A}$ every morphism possesses a kernel and if every kernel is an inflation.

\vspace{3pt}

\item[(iv)] The conflation category $(\mathcal{A},\CC)$ is \emph{exact}, if \hypref{R0}\hspace{2pt}--\hspace{1pt}\hypref{R3} as well as the dual axioms \textcolor{RoyalBlue}{\textbf{L0}}\label{L2}\hspace{1pt}--\hspace{1pt}\textcolor{RoyalBlue}{\textbf{L3}} all hold.\diam{}
\end{compactitem}
\end{dfn}

\begin{rmk}\label{AN-6}\begin{myitemize}\setlength{\itemindent}{-10pt}\item[(i)] In the notation of \cite{HR19,HR19a,HR24} the axioms \hypref{R1}\hspace{1pt}--\hspace{1pt}\hypref{R3} are defined as above, but axiom \hypref{R0} is in these papers called \textbf{R0}$^{\boldsymbol{\ast}}$. The definition with \hypref{R0} as above guarantees that all split kernel-cokernel pairs are conflations, cf.~\cite[Prop 2.6]{HR24} and \cite[Rmk 2.9]{HKRW}. 

\vspace{3pt}\setlength{\itemindent}{0pt}

\item[(ii)] In the notation of \cite[Dfns 3.1 and 3.2]{BC13} the axioms in Definition \ref{DEFLEXCAT} correspond to (R0$^*$)$^{\operatorname{op}}$, R1$^{\operatorname{op}}$, R2$^{\operatorname{op}}$ and R3$^{\operatorname{op}}$, respectively.

\vspace{3pt}\setlength{\itemindent}{0pt}

\item[(iii)] In \cite{HR19} the term `right exact' is used synonymously for `deflation-exact', while in \cite{BC13} deflation-exact categories are called `left exact'. Rump \cite{RumpMax} also uses the term `left exact' but includes \hypref{R3} in the definition.

\vspace{3pt}\setlength{\itemindent}{0pt}

\item[(iv)] By Keller \cite{Keller90} the axioms \hypref{R0}\hspace{2pt}--\hspace{1pt}\hypref{R2} and \hypref{L2} suffice to define an exact category in the sense of Quillen, cf.~\cite[Rmk 2.9]{HKRW}.\diam{}
\end{myitemize}
\end{rmk}

Recall that an additive category is \emph{idempotent complete} or \emph{karoubian} if every idempotent, i.e., every $p\colon X\rightarrow X$ with $p^2=p$, has a kernel. The next result was proven initially for preabelian categories in \cite{SW} and generalized to so-called weakly idempotent complete categories in \cite{Crivei}. For our purposes in the current article the following version will though be sufficient.

\begin{thm}\cite[Thm 3.5]{Crivei}\label{PROP-MAX} Let $\mathcal{A}$ be a karoubian category. Then the stable kernel-cokernel pairs define an exact structure on $\mathcal{A}$. Moreover, this is the maximal exact structure on $\mathcal{A}$.\diam{}
\end{thm}

The following is a straightforward generalization of Theorem \ref{PROP-MAX} and a version of it has been noted already in \cite[Prop 9.2]{HR24}. It indeed holds true for weakly idempotent complete categories, but we formulate it only in the karoubian case as this is the setting which is of interest in the remainder.

\begin{thm}\label{PROP-MAX-D} Let $\mathcal{A}$ be a karoubian category. Let $\mathbb{D}$ be the class of all kernel-cokernel pairs $(f,g)$ in which $g$ is a semistable cokernel. Then $(\mathcal{A},\mathbb{D})$ is strongly deflation-exact. Moreover, $\mathbb{D}$ is the maximal deflation-exact structure on $\mathcal{A}$.
\end{thm}
\begin{proof} In order to see that \hypref{R0} holds let $X,T\in\mathcal{A}$ be given. Then the pullback of $X\rightarrow 0$ along $T\rightarrow0$ is the projection $p\colon X\oplus T\rightarrow T$. Let  $i\colon X\rightarrow X\oplus T$ be the inclusion. Then $p=\cok i$ which means that $X\rightarrow 0$ is a semistable cokernel and thus $X\rightarrow X\rightarrow 0$ belongs to $\mathbb{D}$. Next we see that \hypref{R1} follows from \cite[Prop 3.1]{Crivei} and  \hypref{R2} follows from \cite[Rmk 2.5]{Crivei}. Finally,  \cite[Prop 3.4]{Crivei} yields \hypref{R3}.
\end{proof}

Let $(\mathcal{A},\EE)$ be an exact category. A full additive subcategory $\mathcal{B}\subseteq\mathcal{A}$ is said to be \emph{fully exact}, if it is extension closed with respect to $\EE$; see e.g.~\cite[Dfn 10.5]{Buehler}. If this is the case, then $(\mathcal{B},\EE\cap\mathcal{B})$ is exact where $\EE\cap\mathcal{B}$ consists of those conflations from $\EE$ whose objects belong to $\mathcal{B}$, see e.g.~\cite[Lem 10.20]{Buehler}. Extension-closedness is however not necessary for $(\mathcal{B},\EE\cap\mathcal{B})$ being exact, cf.~the characterization given in \cite[Thm 2.6]{DS12}.

%

\smallskip
\section{LB-spaces}\label{SEC-2}

For the rest of the article let $\mathbb{K}$ be either the field of real numbers or the field of complex numbers. All vector spaces are vector spaces over $\mathbb{K}$. We denote by $\HDLCS$ the category of all Hausdorff locally convex topological vector spaces (short: lcs). If $E$ is a lcs and $F\subseteq E$ a linear subspace, then, if nothing else is said, we understand $F$ endowed with the subspace topology; furthermore, if $F$ is closed, then we understand $E/F$ endowed with the locally convex quotient topology. Given a linear and continuous map $f\colon X\rightarrow Y$ between lcs, the aforementioned applies in particular to $\ran f:=f(X)$, $\overline{f(X)}$, $f^{-1}(0)$, $X/f^{-1}(0)$ and $Y/\overline{f(X)}$. The notation $\ker f$, $\cok f$, $\im f$ and $\coim f$ will be used strictly in the sense of category theory and we will always indicate in which category we take a kernel, cokernel, image or coimage. In order to avoid misunderstandings, we shall occasionally use a prefix and write, e.g., $\mathcal{A}$-kernel, $\mathcal{A}$-semistable kernel, etc, if $\mathcal{A}\subseteq\HDLCS$ is given. In the category $\HDLCS$ and its full subcategories we  use the words `morphism', `linear and continuous map', or simply `map' synonymously. The symbol `$\hookrightarrow$' between lcs stands for the inclusion map of a linear, but not necessarily topological, subspace. For a lcs $E$ we denote its continuous dual space by $E'$ and we write $(E,\w):=(E,\sigma(E,E'))$ for $E$ being endowed with its weak topology.

\begin{dfn}\label{DFN-LB} An \emph{LB-space} is a lcs $(X,\tau)$ such that there exists a sequence of Banach spaces $X_0\rightarrow X_1\rightarrow\cdots$ with linear and continuous structure maps, such that $X=\ind_{n\in\mathbb{N}}X_n$ is the locally convex inductive limit of the sequence. By $\LB\subseteq\HDLCS$ we denote the full subcategory formed by all LB-spaces.\diam{}
\end{dfn}

If $X=\ind_{n\in\NN}X_n$ is an LB-space, then by Grothendieck's factorization theorem we may assume w.l.o.g.\ that the structure maps are inclusions of subspaces. We then write $X_0\hookrightarrow X_1\hookrightarrow\cdots$ and we may think of the LB-space as the union $X=\bigcup_{n=0}^{\infty}X_n$ endowed with the finest locally convex topology $\tau$ that makes all inclusion maps $X_n\hookrightarrow (X,\tau)$ continuous.

\begin{prop}\label{LB-PROP}\cite[Rmk 3.1.1]{DS16} The category $\LB$ is preabelian. More precisely:\vspace{2pt}
\begin{compactitem}

\item[(i)] The kernel of $f\colon X\rightarrow Y$ is the inclusion map $f^{-1}(0)^{\flat}\rightarrow X$, where $f^{-1}(0)^{\flat}$ is the LB-space $\ind_{n\in\mathbb{N}}f^{-1}(0)\cap X_n$.\vspace{3pt}

\item[(ii)] The cokernel of $f\colon X\rightarrow Y$ is the quotient map $Y\rightarrow Y/\overline{f(X)}$, where the quotient space is endowed with the locally convex quotient topology (which defines an LB-space).\vspace{3pt}

\item[(iii)] $f\colon X\rightarrow Y$ is a kernel iff $f$ is injective and has a closed range.

\vspace{3pt}

\item[(iv)] $g\colon Y\rightarrow Z$ is a cokernel iff $g$ is surjective.

\vspace{3pt}

\item[(v)] $X\stackrel{f}{\rightarrow}Y\stackrel{g}{\rightarrow}Z$ is a kernel-cokernel pair iff the sequence is \emph{algebraically exact}, i.e., iff $f$ is injective, $g$ is surjective and $f(X)=g^{-1}(0)$ holds as vector spaces.

\vspace{3pt}

\item[(vi)] $f$ is monic iff $f$ is injective.

\vspace{3pt}

\item[(vii)] $g$ is epic iff $g$ has a dense range.\diam{}

\end{compactitem}
\end{prop}

The $\flat$-topology is the ultrabornological topology associated with the subspace topology, see \cite[Chapter 13.3]{Jarchow} or \cite[Dfn 6.2.4]{BPC}, and it may be strictly finer than the subspace topology. First examples in which this happens are due to Grothen\-dieck \cite{G}, see Example \ref{EX-GROTH} below. Notice that in these examples the defining sequences of Banach spaces are \emph{strict} in the sense that $E_n\subseteq E_{n+1}$ is a topological subspace for every $n\in\NN$; the limit space is hence complete.

\begin{ex}\label{EX-GROTH}\cite[Ex 8.6.12 and 8.6.13]{BPC}\vspace{2pt}

\begin{compactitem}

\item[(i)] There exists a complete LB-space $(E,\tau)=\ind_{n\in\NN}E_n$ and a closed subspace $H\subseteq E$ such that $(H,\sigma)'\not=(H,\tau)'$ holds, where $(H,\sigma)=\ind_{n\in\NN}H\cap E_n$.

\vspace{4pt}

\item[(ii)] There exists a complete LB-space $(E,\tau)=\ind_{n\in\NN}E_n$ and a closed subspace $H\subseteq E$ such that $(H,\sigma)'=(H,\tau)'$, but $(H,\sigma)\not=(H,\tau)$, where $\sigma$ is defined  as in (i).\diam{}

\end{compactitem}
\end{ex}

\smallskip

From Example \ref{EX-GROTH}(i) we may derive in which way precisely $\LB$ fails to be (quasi)abelian.

\begin{rmk}\label{PAR-1}\cite[Proof of Prop 14]{Wegner17} In $\LB$, the \emph{parallel} of a morphism $f\colon X\rightarrow Y$, i.e., the canonical map $\bar{f}\colon\coim f\rightarrow \im f$ between coimage and image, is given by
$$
\bar{f}\colon X/f^{-1}(0)\rightarrow \overline{f(X)}^{\hspace{1pt}\flat}
$$
and it is always a monomorphism but in general not an epimorphism. An example for this effect is as follows. In the situation of Example \ref{EX-GROTH}(i) we pick $u\in(H,\sigma)'\backslash(H,\tau)'$ and consider the inclusion map $f\colon u^{-1}(0)^{\flat}\rightarrow E$. By the choice of $u$ we know that $u^{-1}(0)\subseteq(H,\sigma)$ is a proper closed subset while $u^{-1}(0)\subseteq(H,\tau)$ is dense. Thus, the inclusion map
\begin{equation*}
\bar{f}\colon u^{-1}(0)^{\flat}\rightarrow(\hspace{1pt}\overline{u^{-1}(0)}^{\scriptscriptstyle(E,\tau)})^{\flat}=(H,\tau)^{\flat}=(H,\sigma)
\end{equation*}
has closed range without being surjective and therefore cannot be epic.\diam{}
\end{rmk}

\begin{rmk}\label{RMK-LB-1}\cite[Proof of Thm 3.4]{HSW} As $\LB$ is preabelian, arbitrary pullbacks and arbitrary pushouts exist. Indeed, if $g\colon Y\rightarrow Z$ and $t\colon T\rightarrow Z$ are morphisms, the pullback of $g$ along $t$ is given by
\begin{equation}\label{PB-00}
\hspace{-39pt}\begin{tikzcd}[column sep=20pt,row sep=21pt]
[t\;\shortminus\hspace{-2pt}g]^{-1}(0)^{\flat}\arrow[d, swap, "p_Y"]\arrow[r, "p_T"]\commutes[\mathrm{PB}]{dr}& T \arrow[d, "t"]\\
Y\arrow[r, swap, "g"]& Z.
\end{tikzcd}
\end{equation}
Assume now that $g$ is a cokernel and $t$ is arbitrary. Then $g$ is surjective and it follows that $p_T$  is surjective, too. Consequently, in $\LB$ pullbacks of cokernels along arbitrary morphisms are again cokernels. Consider morphisms $f\colon X\rightarrow Y$ and $t\colon X\rightarrow T$. Then the pushout of $f$ along $t$ is given by
\begin{equation}\label{PO-00}
\begin{tikzcd}
X\arrow{r}{f}\arrow{d}[swap]{t}\commutes[\mathrm{PO}]{dr} & Y \arrow{d}{q_Y}\\[4pt]
T \arrow{r}[swap]{q_T} & \frac{Y\oplus T}{\overline{\operatorname{ran}[f\:\sm\hspace{0pt}t]^{\operatorname{T}}}}.
\end{tikzcd}
\end{equation}
Now let us give an example of a pushout diagram \eqref{PO-00} in which $f$ is a kernel but $q_T$ is not. Indeed, in the situation of Example \ref{EX-GROTH}(i) let $f\colon (H,\sigma)\rightarrow(E,\tau)$ be the inclusion map, let $u$ be as in the previous paragraph and let $t\colon(H,\sigma)\rightarrow(H,\sigma)/u^{-1}(0)$ be the quotient map. Then $f$ is a kernel, but
$$
q_T(x+u^{-1}(0))=\bigl[\begin{smallmatrix}0\\x+u^{-1}(0)\end{smallmatrix}\bigr]+\overline{\bigl\{\bigl[\begin{smallmatrix}y\\-(y +u^{-1}(0))\end{smallmatrix}\bigr]|\,y\in H\bigr\}}=0
$$
holds for every $x\in X$ while $(H,\sigma)/u^{-1}(0)\not=0$.\diam{}
\end{rmk}

Remarks \ref{PAR-1} and \ref{RMK-LB-1}  mean that $\LB$ is left quasiabelian (left almost abelian in \cite[p.~167f]{Rump01}), but not right semiabelian and thus in particular not quasiabelian. We shall use these notions only occasionally in the remainder, but we refer to \cite{KW} and \cite[Section 2]{HSW} for more details.

\smallskip

Let us now consider conflation structures on $\LB$. Observe firstly that, if $(f,g)$ is a kernel-cokernel pair, $\ran f=g^{-1}(0)$ is automatically closed, but not automatically an LB-space in the induced topology. In particular, $f$ is in general not an isomorphism onto its range by Example \ref{EX-GROTH}(ii) and may even fail to be a weak isomorphism onto its range by Example \ref{EX-GROTH}(i). On the other hand, $g$ is always open due to the open mapping theorem.

\begin{thm}\label{LB-confl-str}\cite[Thm 9.1]{HKRW}, \cite[Prop 3.1.2]{DS16}, \cite[Prop 3.3]{DS12} In the category of LB-spaces the class of all kernel-cokernel pairs\vspace{-3pt}
\begin{equation*}
\begin{aligned}
\Call&:=\bigl\{(f,g)\in\LB\times\LB\;|\:f=\ker g\:\text{and } g=\cok f\hspace{1pt}\bigr\}\hspace{89pt}\\[-1pt]
&\phantom{:}=\bigl\{X\stackrel{f}{\rightarrow}Y\stackrel{g}{\rightarrow}Z\in\LB\text{ is algebraically exact}\hspace{1pt}\bigr\}\\[3.5pt]
&\phantom{:}=\Dmax
\end{aligned}
\end{equation*}
is a strongly deflation-exact conflation structure with admissible kernels and it is the maximal conflation structure on $\LB$. Moreover, we have
\begin{equation*}
\begin{aligned}
\Emax&:=\bigl\{(f,g)\in\Call\;|\;(f,g)\;\text{is a stable kernel-cokernel pair}\hspace{1pt}\bigr\}\hspace{70pt}\\
&\phantom{:}=\bigl\{(f,g)\in\Call\;|\;(\ran f,\w)=(g^{-1}(0)^{\flat},\w)\text{ as topological spaces}\hspace{1pt}\bigr\}\hspace{71pt}
\end{aligned}
\end{equation*}
which is the maximal exact structure on $\LB$. Finally,
\begin{equation*}
\begin{aligned}
\Etop&:=\bigl\{(f,g)\in\Call\;|\;\ran f=g^{-1}(0)^{\flat}\text{ as topological spaces}\hspace{1pt}\bigr\}\hspace{27.5pt}
\end{aligned}
\end{equation*}
is another exact structure on $\LB$\emph{;} $\Etop\subset\Emax\subset\Dmax$ holds with strict inclusions.\diam{}
\end{thm}

The above is closely related to the following classic notions: Let $E=\ind_{n\in\NN}E_n$ be an LB-space. A closed subspace $F\subseteq E$ is said to be a \emph{limit subspace} if on $F$ the topology induced by $E$ and the $\flat$-topology coincide. It is said to be a \emph{well-located subspace} if these two topologies lead to the same dual space, or, equivalently, if the weak topologies coincide, see \cite[Section 8.6]{BPC}.

\begin{rmk}\label{RMK-LB-LAST}For later use we note the following characterization of semi\-stable (co)ker\-nels in $\LB$. The first statement we explained in Remark \ref{RMK-LB-1}, the second is due to Dierolf \cite{DPriv}, and follows from Lemma \ref{LEM-D} below.\vspace{3pt}
\begin{compactitem}

\item[(i)] In $\LB$ every cokernel is semistable.\vspace{3pt}

\item[(ii)] In $\LB$ a kernel $f\colon X\rightarrow Y$ is semistable if and only if $\ran f\subseteq Y$ is well-located.\diam{}

\end{compactitem}
\end{rmk}

Traditionally, an algebraically exact sequence $X\stackrel{\scriptscriptstyle f}{\rightarrow}Y\stackrel{\scriptscriptstyle g}{\rightarrow}Z$ of lcs is called \emph{topologically exact} if $f$ is an isomorphism onto its range and $g$ is open, see e.g.~\cite[Section 26]{MV}. Indeed, the topologically exact sequences are precisely the kernel-cokernel pairs in the quasiabelian category $\HDLCS$, see \cite{Prosmans}. The exact structure $\Etop$ is thus the restriction of the natural\,(=\hspace{0.5pt}maximal) exact structure of $\HDLCS$ to $\LB$; however $\LB\subseteq\HDLCS$ is \emph{not} extension closed, see \cite[Table on p.~2114]{DS12} or \cite[Section 2.3]{FW11}. Let us finally point out that not every cokernel in $\LB$ is an $\Emax$-inflation as its kernel might fail to be semistable. On the other hand, every semistable kernel is an $\Emax$-inflation.

\medskip

In the following two sections we investigate the properties of the subcategories of $\LB$ formed by the regular and complete LB-spaces, respectively. In particular, we study their natural conflation structures. We shall use letters $\CC$ for classes of all kernel-cokernel pairs, $\EE$ for exact structures and $\mathbb{D}$ for deflation-exact structures; $\CC$ and $\mathbb{D}$ falling together, as it happens in $\LB$, will turn out to be a coincidence and not a rule.

\smallskip

\section{Regular LB-spaces}\label{SEC-3}

Let $X=\ind_{n\in\NN} X_n$ be the inductive limit of the sequence of Banach spaces $X_0\rightarrow X_1\rightarrow\cdots$  with arbitrary structure maps and let $i_n\colon X_n\rightarrow X$ be the natural maps. Then for every bounded set $B\subseteq X_n$ the set $i_n(B)\subseteq X$ is bounded in $X$. Regularity describes those spaces in which all bounded sets arise in the aforementioned way.

\begin{dfn}\label{DFN-LBR} An LB-space $X=\ind_{n\in\NN} X_n$, w.l.o.g.\ given by a sequence where the structure maps are inclusions, is \emph{regular} if for any bounded set $B\subseteq X$ there exists $n\in\NN$ such that $B\subseteq X_n$ holds and $B$ is bounded in $X_n$. We denote by $\LBr\subseteq\LB$ the full subcategory consisting of the regular LB-spaces.\diam{}
\end{dfn}

It is nontrivial to see that there are LB-spaces which actually fail to be regular: The first such space was defined by K\"othe \cite[p.~326]{K48}, who in \cite[p.~629]{K50} states and proves that it is not (sequentially) complete. Without using the word `regular', Grothendieck \cite[Chapitre I, p.~12--13]{G55} refers to K\"othe's space when he points out that countable inductive limits of Banach spaces may fail to be complete or fail to be regular.

\smallskip

Regularity can be defined in the obvious way for arbitrary sequences $X_0\rightarrow X_1\rightarrow \cdots$ of Banach spaces that define an LB-space $X$. By Grothendieck's factorization theorem $X$ is then regular in the sense of Definition \ref{DFN-LBR} if and only if one, or equivalently, every defining sequence is regular.

\smallskip

\begin{rmk}\label{Stand-RES} While it is straightforward to check that regularity inherits to closed subspaces furnished with the $\flat$-topology, it in general does not inherit to $\LB$-cokernels. Indeed, for any (regular or not) LB-space $X=\ind_{n\in\NN}X_n$ we have the `standard resolution', see \cite[Equation (1) on p.~58]{Vogt92},
\begin{equation}\label{RES}
\Bigosum{n\in\mathbb{N}}{}X_n\stackrel{d}{\longrightarrow}\Bigosum{n\in\mathbb{N}}{}X_n\stackrel{\sigma}{\longrightarrow}X
\end{equation}
with $d((x_n)_{n\in\NN})=(x_n-x_{n-1})_{n\in\NN}$, $x_{-1}=0$, and $\sigma((x_n)_{n\in\NN})=\sum_{n\in\NN}x_n$. The above is a kernel-cokernel pair in $\LB$ and the direct sums are regular LB-spaces w.r.t.\ the spectrum $X_0\hookrightarrow X_0\oplus X_1\hookrightarrow X_0\oplus X_1\oplus X_2\hookrightarrow\cdots$. In particular, every non-regular LB-space is the quotient of a regular one.\diam{}
\end{rmk}

It is well-known that regularity of LB-spaces is equivalent to so-called Mackey completeness (or local completeness), see e.g.~\cite[Lem 7.3.3]{BPC}: Let $E$ be a lcs and $B$ be a disk, i.e., a bounded and absolutely convex set. Then $E_B:=\spann B$ endowed with the Minkowski functional of $B$ is the so-called auxiliary normed space. We say that a sequence $(x_k)_{k\in\NN}\subseteq E$ is Mackey convergent, resp.\ Mackey Cauchy, if there exists a disk $B\subseteq E$ such that $(x_k)_{k\in\NN}\subseteq E_B$ is convergent, resp.\ Cauchy. The space $E$ is Mackey complete if every Mackey Cauchy sequence is Mackey convergent. If $E$ is not Mackey complete, we consider $E$ first as a subspace of its (usual) completion $\widehat{E}$ and then form the intersection of all Mackey complete subspaces of $\widehat{E}$ containing $E$, i.e., we put
$$
E^{\sharp}\hspace{4pt}=\hspace{-4pt}\mathop{\textstyle\bigcap}_{\stackrel{E\subseteq F\subseteq \widehat{E}}{\stackrel{F \text{ Mackey}}{\scriptscriptstyle\text{complete}}}}\hspace{-3pt}F.
$$ 
The latter, endowed with the topology induced by $\widehat{E}$, is then a Mackey complete space which we refer to as the \emph{Mackey completion}, or \emph{mackeyfication} of $E$, see \cite[Dfn 5.1.21]{BPC}. Given a linear and continuous map $f\colon E\rightarrow F$ between lcs, this map extends uniquely to a linear and continuous map $f^{\sharp}\colon E^{\sharp}\rightarrow F^{\sharp}$, see \cite[Prop 5.1.25]{BPC}. As it turns out, the mackeyfication of an LB-space is a regular LB-space, cf.~Doma\'nski \cite[p.~61]{Do98}, and we get the following.

\begin{lem}\label{DOM-LEM} Mackeyfication defines a functor $\mac\colon\LB\rightarrow\LBr$, $X\mapsto X^{\sharp}$, $f\mapsto f^{\sharp}$, which is left adjoint to the inclusion functor $\operatorname{inc}\colon\LBr\rightarrow\LB$. More precisely, for $X\in\LB$ and $Y\in\LBr$, the restriction $f\mapsto f|_X$ yields natural isomorphisms:
$$
\Hom_{\LBr}(\mac X,Y)\cong\Hom_{\LB}(X,\inc Y).
$$
\end{lem}
\begin{proof} As Doma\'nski omitted a proof, we first fill in the details on why $\mac\colon\LB\rightarrow\LBr$ is well-defined. Let $X=\ind_{n\in\NN}X_n$ be an LB-space. By \cite[Props 6.2.8 and 6.1.8]{BPC} the Mackey completion is bornological and can be written as the inductive limit of all auxiliary normed spaces, i.e.,
\begin{equation}\label{BOR}
X^{\sharp}=\mathop{\ind}_{\stackrel{\scriptscriptstyle B\subseteq X^{\sharp}}{\scriptscriptstyle\hspace{-2pt}\phantom{|^i}\text{disk}\phantom{|^i}}}(X^{\sharp})_B.
\end{equation}
Moreover, as an LB-space, $X$ is a DF-space, see \cite[Cor 8.3.20]{BPC}. The proofs of \cite[Prop 8.3.16(i)--(iii)]{BPC}, or the comments in  \cite[p.~79]{Klaus}, thus yield that the system $(\overline{B_n}^{\scriptscriptstyle X})_{n\in\NN}$, formed by the closures of the step's unit balls $B_n\subseteq X_n$, is a fundamental sequence of bounded sets for $X$, i.e., we have
\begin{equation}\label{FSB}
\forall\:B\subseteq X\text{ bounded}\;\exists\:n\in\NN,\,\lambda>0\colon B\subseteq \lambda\overline{B_n}^{\scriptscriptstyle X}.
\end{equation}
We claim now that $(\overline{B_n}^{\scriptscriptstyle X^{\sharp}})_{n\in\NN}$ is a fundamental system of bounded sets for $X^{\sharp}$. Let $C\subseteq X^{\sharp}$ be bounded. Then $C\subseteq\widehat{X}$ is bounded as well and by \cite[Cor 8.3.17(i)]{BPC} there exists a bounded $B\subseteq X$ such that $C\subseteq\overline{B}^{\scriptscriptstyle\widehat{X}}$ holds. By \eqref{FSB} we find $n\in\NN$ and $\lambda>0$ such that 
$$
C=C\cap X^{\sharp} \subseteq\overline{B}^{\scriptscriptstyle\widehat{X}}\cap X^{\sharp}=\overline{B}^{\scriptscriptstyle X^{\sharp}}\subseteq \lambda \overline{B_n}^{\scriptscriptstyle X^{\sharp}}
$$
holds, which establishes the claim and implies that in \eqref{BOR} we get the same limit space if we take the inductive limit only over the following countable subsystem
$$
X^{\sharp}=\mathop{\ind}_{n\in\NN}(X^{\sharp})_{\scalebox{0.8}{$\overline{B_n}^{\scriptscriptstyle X^{\sharp}}$}}.
$$
Here, the auxiliary normed spaces are Banach since the $\overline{B_n}^{\scriptscriptstyle X^{\sharp}}$ are by construction closed disks in a Mackey complete space, see \cite[Prop 5.1.6]{BPC}. Thus, $X^{\sharp}$ is an LB-space which is regular by construction, or by applying \cite[Lem 7.3.3]{BPC}.

\smallskip

As we have seen, an LB-space $X$ is regular if and only if it is Mackey complete if and only if $X=X^{\sharp}$ holds. From this the adjunction follows immediately.
\end{proof}

Using mackeyfication, we can now prove the existence of cokernels in $\LBr$. Below one may conclude (i) from \cite[Prop 4.5.15(i)]{Riehl} and (ii) from \cite[Thm 4.5.3]{Riehl}, or check the statements explicitly.

\begin{prop}\label{LBrKERCOK} The category $\LBr$ is preabelian. More precisely:\vspace{3pt}
\begin{compactitem}
\item[(i)] The kernel of $f\colon X\rightarrow Y$ is the inclusion $f^{-1}(0)^{\flat}\rightarrow X$, \vspace{0pt}
\item[(ii)] The cokernel of $f\colon X\rightarrow Y$ is the composition $i\circ q\colon Y\rightarrow\bigl(Y/\overline{f(X)}\hspace{1pt}\bigr)^{\!\sharp}$ of quotient and inclusion map.\vspace{3pt}

\item[(iii)] $f\colon X\rightarrow Y$ is a kernel iff $f$ is injective and has a closed range.\vspace{2pt}

\item[(iv)] $g\colon Y\rightarrow Z$ is a cokernel iff $g$ is relatively open and its range is dense.\vspace{0.5pt}

\item[(v)] $X\stackrel{f}{\rightarrow}Y\stackrel{g}{\rightarrow}Z$ is a kernel-cokernel pair iff $f$ is injective, $f(X)=g^{-1}(0)$ holds as vector spaces and $g$ is relatively open and its range is dense.

\vspace{2pt}

\item[(vi)] $f$ is monic iff $f$ is injective.

\vspace{2pt}

\item[(vii)] $g$ is epic iff $g$ has a dense range.
\end{compactitem} 
\end{prop}
\begin{proof}(i)\hspace{1pt}--\hspace{1pt}(iii) are straightforward; (vi)\hspace{1pt}--\hspace{1pt}(vii) follow immediately; (v) follows and from (iii)\hspace{1pt}--\hspace{1pt}(iv) and in (iv) only the suffiency needs to be checked. To this end let $g\colon Y\rightarrow Z$ be open onto its range and let the range be dense. We consider the $\LBr$-kernel $i\colon g^{-1}(0)^{\flat}\rightarrow Y$ of $g$ and claim that $g=\cok i$ holds in $\LBr$. By definition we have $g\circ i=0$. Let $h\colon Y\rightarrow H$ in $\LBr$ be given with $h\circ i=0$. The map $g$ induces an isomorphism $\bar{g}\colon Y/g^{-1}(0)\rightarrow (g(Y),\uptau_Z)$ due to the assumed openness. Denote by  $q\colon Y\rightarrow Y/g^{-1}(0)$ the quotient map and by $j\colon g(Y)\rightarrow Z$ the inclusion. Then we have $g=j\circ\bar{g}\circ q$. Now we obtain, from left to right, unique dashed maps as in the following diagram
\begin{equation*}
\begin{tikzcd}[column sep=32pt,row sep=17pt]
& & Y/g^{-1}(0)\arrow{r}{\bar{g}}\arrow{r}[swap]{\sim}\arrow[swap,dashed]{dd}{}& (g(Y),\uptau_Z)\arrow{r}{j}\arrow[swap,dashed]{ddl}{} &Z\arrow[bend left=12, dashed]{ddll}{}\\[-7pt]
g^{-1}(0)^{\flat}\arrow{r}{i}\arrow[bend left=20]{urr}{0}\arrow[bend right=20, swap]{drr}{0} & Y \arrow[swap]{dr}{h}\arrow{ur}{q} & \\[-7pt]
&  & H
\end{tikzcd}
\end{equation*}
by using firstly that $q=\HDLCS\hspace{1pt}\text{-}\hspace{-1pt}\cok i$ holds, secondly that $\overline{g}$ is an isomorphism of lcs, and finally by employing that $g(Y)\subseteq Z$ is dense.
\end{proof}

Let us highlight that by the above cokernels in $\LBr$ are in general not surjective.

\begin{rmk}\label{EX-LBr-1} The parallel of a morphism $f\colon X\rightarrow Y$ in $\LBr$ is 
$$
\bar{f}\colon(X/f^{-1}(0))^{\sharp}\longrightarrow\overline{f(X)}^{\flat},
$$
i.e., the mackeyfication of the injection $X/f^{-1}(0)\rightarrow \overline{f(X)}^{\flat}$. The parallel is in general neither monic nor epic, as we now show.

\smallskip

\textcircled{1} Let first $f\colon u^{-1}(0)^{\flat}\rightarrow E$, with $u\in(H,\sigma)'\backslash(H,\tau)'$ from Example \ref{EX-GROTH}(i), be as in Remark \ref{PAR-1}. Observe that $E$ in Example \ref{EX-GROTH}(i) is regular and thus $f$ is indeed a morphism in $\LBr$. As $f$ is injective, and $u^{-1}(0)^{\flat}$ as well as $(H,\sigma)$ are regular, we get the same parallel
$$
\bar{f}\colon u^{-1}(0)^{\flat}\rightarrow (H,\sigma)
$$
in $\LBr$ as in $\LB$, cf.~Remark \ref{PAR-1}, where we already saw that $\bar{f}$ does not have a dense range.

\smallskip

\textcircled{2} It remains to give an example where the parallel is not injective. For this we adapt a method of Sieg \cite[Proof of Prop 3.1.6]{SiegThesis}: Let $U$ be a closed subspace of a regular LB-space $X$ such that $X/U$ is not regular and put $Y:=(X/U)^{\sharp}$. Let $c\colon X\rightarrow Y$ be the composition of quotient map and inclusion. We pick $y_0\in Y\backslash c(X)$ and define
$$
f\colon X\oplus\KK\rightarrow Y,\;(x,\lambda)\mapsto c(x)-\lambda y_0
$$
which is a morphism in $\LBr$ with $f^{-1}(0)=U\oplus\{0\}$ and thus has coimage $(X\oplus\KK/U\oplus\{0\})^{\sharp}=Y\oplus\KK$. As $c$ has dense range by construction, we further get that $\overline{f(X)}=Y$. The parallel of $f$ is thus given by
$$
\bar{f}\colon Y\oplus\KK\rightarrow Y,\;(y,\lambda)\mapsto y-\lambda y_0
$$
with $\bar{f}(y_0,1)=0$.\hfill\scalebox{1.4}{$\diamond$}
\end{rmk}

By the above, $\LBr$ is neither left nor right semi-abelian and has thus by \cite[Cor 1 on p.\ 169]{Rump01} kernels as well as cokernels which are not semistable. In contrast to $\LB$, the class
$$
\Callreg:=\bigl\{(f,g)\in\LBr\times\LBr\:\bigl|\:f=\ker g \text{ and } g=\cok f\bigr\}
$$
of all kernel-cokernel pairs in $\LBr$ is thus \emph{neither} a deflation-exact structure \emph{nor} an inflation-exact structure. Nevertheless, the next result characterizes the semistable cokernels of $\LBr$.

\begin{prop}\label{LBr-SUR} In $\LBr$ a morphism is a semistable cokernel iff it is surjective.
\end{prop}
\begin{proof}\textquotedblleft{}$\Longrightarrow$\textquotedblright{} The implication follows from the indirect argument in \cite[Prop 2.2.3]{SiegThesis}. Alternatively, we may also argue directly: Let $g\colon Y\rightarrow Z$ be a semistable cokernel and $z_0\in Z$. Consider $t\colon\mathbb{K}\rightarrow Z$, $\lambda\mapsto\lambda z_0$, and form the $\LBr$-pullback $(P,p_{\mathbb{K}},p_{Y})$ of $g$ along $t$. Then $p_{\KK}$ is a cokernel and in view of its one-dimensional codomain, $p_{\KK}$ has to be surjective. We thus find $x\in P$ with $p_{\KK}(x)=1$ and therefore $g(p_Y(x))=z_0$.

\smallskip

\textquotedblleft{}$\Longleftarrow$\textquotedblright{} If we form a pullback of a surjective morphism $g\colon Y\rightarrow Z$ along an arbitrary morphism $t\colon T\rightarrow Z$ in $\LBr$, we get exactly the same diagram as in \eqref{PB-00}. We see that $p_T$ will be surjective and an $\LBr$-cokernel by Lemma \ref{LBrKERCOK}(iv).
\end{proof}

\begin{thm}\label{C-REG-SUR} On $\LBr$ the following defines a strongly deflation-exact conflation structure and the maximal deflation-exact structure:
$$
\Cregsur:=\bigl\{(f,g)\in\Callreg\:\big|\: g\;\text{is semistable}\hspace{0.5pt}\bigr\}= \Call\cap\LBr.
$$
\end{thm}
\begin{proof} By Theorem \ref{PROP-MAX-D} the class $\Cregsur$ as defined above has the properties stated in the theorem. The equality follows from Propositions \ref{LBr-SUR}, \ref{LBrKERCOK}(v) and \ref{LB-PROP}(v). 
\end{proof}

\begin{rmk}\begin{myitemize}\setlength{\itemindent}{-10pt}\item[(i)] The conflation category $(\LBr,\Cregsur)$ does not have admissible kernels in the sense of Definition \ref{DEFLEXCAT}(iii). Indeed, if $g$ is a non-surjective cokernel, then $\ker g$ is not a $\Cregsur$-inflation.

\setlength{\itemindent}{0pt}\vspace{3pt}

\item[(ii)] That $(\LBr,\Cregsur)$ is deflation-exact and that not all kernels are inflations can alternatively be derived from the dual of \cite[Cor 4.5]{BC13}: By Lemma \ref{DOM-LEM}, $\LBr\subseteq\LB$ is a reflective subcategory of a left quasiabelian category which is not right quasiabelian.\diam{}
\end{myitemize}
\end{rmk}

Let us now look at exact structures. As $\LBr$ is preabelian, Theorem \ref{PROP-MAX} yields that the stable kernel-cokernel pairs form the maximal exact structure on $\LBr$. In order to characterize its members, we need the following four lemmas; the first is well-known, we only recall it for convenient reference later on. 

\begin{lem}\label{AN-2a}\cite[Lem 2.1]{BC13} Let $\mathcal{A}$ be an additive category. Let $f\colon X\rightarrow Y$ be a morphism which has a cokernel $c\colon Y\rightarrow Z$. Let $t\colon X\rightarrow T$ be a morphism such that the pushout of $f$ along $t$ exists. Then the pushout property gives rise to the commutative diagram
 \begin{equation*}
\begin{tikzcd}
X\arrow{d}[swap]{t}\arrow{r}{f}\commutes[\mathrm{PO}]{dr}&Y\arrow{r}{c}\arrow{d}{q_Y} & Z \arrow[equal]{d}\\[4pt]
T\arrow[swap]{r}{q_T}\arrow[rr,bend right =30,swap, "0"] &Q \arrow[dashed]{r}[swap]{c'} & Z.
\end{tikzcd}
\end{equation*}
in which $c'$ is a cokernel of $q_T$.\diam{}
\end{lem}

The second lemma is a slight modification of results by Dierolf \cite{DPriv} that were used to prove the explicit description of $\Emax$, cf.~Theorem \ref{LB-confl-str}. For the convenience of the reader we include a proof.

\begin{lem}\label{LEM-D}\emph{[\hspace{1pt}essentially due to B.~Dierolf]} Consider a pushout diagram
\begin{equation}\label{LEM-D-PO}
\begin{tikzcd}
X\arrow{r}{f}\arrow{d}[swap]{t}\commutes[\mathrm{PO}]{dr} & Y \arrow{d}{q_Y}\\[4pt]
T \arrow{r}[swap]{q_T} & \frac{Y\oplus T}{\overline{\operatorname{ran}[f\:\sm\hspace{.5pt}t]^{\scriptscriptstyle\operatorname{T}}}}
\end{tikzcd}
\end{equation}
in the category $\HDLCS$. Then the following holds true:\vspace{3pt}
\begin{compactitem}

\item[(i)] For any diagram \eqref{LEM-D-PO} we have: If $f$ has closed range, then $q_T$ has closed range.

\vspace{3pt}

\item[(ii)] Let $\Class\subseteq\operatorname{Ob}(\HDLCS)$ be a collection of objects containing the field $\KK$. Let $f\colon X\rightarrow Y$ be a morphism in $\HDLCS$ that is injective and has a closed range. Identify, as linear spaces, $X$ with $f(X)\subseteq Y$. Then the following are equivalent:
\begin{compactitem}\vspace{2pt}
\item[(a)] We have $(X,\tau_X)'=(X,\tau_Y)'$.\vspace{3pt}
\item[(b)] For any $t\colon X\rightarrow T$ with $T\in\Class$ the map $t\colon(X,\tau_Y)\rightarrow T$ has a closed graph.\vspace{3pt}
\item[(c)] For any $t\colon X\rightarrow T$ with $T\in\Class$ the map $[f\:\sm{}\hspace{-1.5pt}t]^{\scriptscriptstyle\operatorname{T}}\colon X\rightarrow Y\oplus T$ has a closed range.\vspace{3pt}
\item[(d)] For any $t\colon X\rightarrow T$ with $T\in\Class$, the map $q_T$ in \eqref{LEM-D-PO} is injective.
\end{compactitem}
\end{compactitem}
\end{lem}
\begin{proof} (i) Since $f(X)\subseteq Y$ is closed, $Y/f(X)$ is Hausdorff. Let $Y\oplus T\rightarrow Y/f(X)$ be the composition of projection and quotient map. The kernel of the latter is $f(X)\oplus T$, in which $\overline{\operatorname{ran}[f\:\sm\hspace{-1.5pt}t]^{\scriptscriptstyle\operatorname{T}}}$ is contained due to the closedness of $f(X)$. Therefore the induced map
$$
\varphi\colon\zfrac{Y\oplus T}{\overline{\operatorname{ran}[f\:\sm\hspace{-1.5pt}t]^{\scriptscriptstyle\operatorname{T}}}}\longrightarrow \zfrac{Y}{f(X)},\;(y,s)+\overline{\operatorname{ran}[f\:\sm\hspace{-1.5pt}t]^{\scriptscriptstyle\operatorname{T}}}\,\mapsto \,y+f(X)
$$
is well-defined, linear and continuous. We see $\varphi\circ q_T=0$ and thus $q_T(T)\subseteq\varphi^{-1}(0)$. Let now $(y,s)+\overline{\operatorname{ran}[f\:\sm\hspace{-1.5pt}t]^{\scriptscriptstyle\operatorname{T}}}\in\varphi^{-1}(0)$ be given. Then $0=\varphi((y,s)+\overline{\operatorname{ran}[f\:\sm\hspace{-1.5pt}t]^{\scriptscriptstyle\operatorname{T}}})=y+f(X)$. We thus find $x\in X$ such that $f(x)=y$ and calculate
\begin{eqnarray*}
(y,s)+\overline{\operatorname{ran}[f\:\sm\hspace{-1.5pt}t]^{\scriptscriptstyle\operatorname{T}}} &=& (y,s)-(f(x),-t(x))+\overline{\operatorname{ran}[f\:\sm\hspace{-1.5pt}t]^{\scriptscriptstyle\operatorname{T}}}\\
&=&(0,s+t(x))+\overline{\operatorname{ran}[f\:\sm\hspace{-1.5pt}t]^{\scriptscriptstyle\operatorname{T}}}
\end{eqnarray*}
which belongs to $q_T(T)$. Therefore, $q_T(T)=\varphi^{-1}(0)\subseteq\frac{Y\oplus T}{\overline{\operatorname{ran}[f\:\sm\hspace{-.5pt}t]^{\scriptscriptstyle\operatorname{T}}}}$ is closed.

\smallskip

(ii) Recall first that $(X,\tau_X)'=(X,\tau_Y)'$ is equivalent to the equality $\sigma(X,(X,\tau_X)')=\sigma(X,(X,\tau_Y)')$ and thus to $(X,\tau_X)\rightarrow(Y,\tau_Y)$ being a weak isomorphism onto its range, see, e.g.~\cite[Chapter 23]{MV}. Now we prove the implications.

\smallskip

\textquotedblleft{}(a)$\Longrightarrow$(b)\textquotedblright{} As previously mentioned we have $\sigma(X,(X,\tau_X)')=\sigma(X,(X,\tau_Y)')=:\sigma(X,X')$. As $t\colon (X,\tau_X)\rightarrow T$ is continuous, $t\colon(X,\sigma(X,X'))\rightarrow (T,\sigma(T,T'))$ is continuous as well. In particular, $\gr(t)\subseteq (X,\sigma(X,X'))\times(T,\sigma(T,T'))$ is closed. By our assumption in (a) the identity
$$
\id_{X\times T}\colon (X,\tau_Y)\times T\rightarrow (X,\sigma(X,X'))\times(T,\sigma(T,T'))
$$
is continuous and thus  $\gr(t)\subseteq (X,\tau_Y)\times T$ is closed, too.

\smallskip

\textquotedblleft{}(b)$\Longrightarrow$(c)\textquotedblright{} Let $(x_{\alpha})_{\alpha\in A}\subseteq X$ be a net such that $f(x_{\alpha})\rightarrow y\in Y$ and $-t(x_{\alpha})\rightarrow s\in T$ both hold. As $f(X)\subseteq Y$ is closed there exists $x\in X$ with $f(x)=y$. Through our identification of $X$ and $f(X)$ as linear spaces, the former means $x_{\alpha}\rightarrow x$ in $(X,\tau_Y)$. As $-t\colon(X,\tau_Y)\rightarrow T$ has closed graph by (b), we obtain $-h(x)=s$. Consequently,
$$
\bigl[\begin{smallmatrix}\phantom{\sm}f\\\sm{}t\end{smallmatrix}\bigr](x_{\alpha})=\bigl[\begin{smallmatrix}\phantom{\sm}f(x_{\alpha})\\\sm{}t(x_{\alpha})\end{smallmatrix}\bigr]\longrightarrow\bigl[\begin{smallmatrix}y\\s\end{smallmatrix}\bigr]=\bigl[\begin{smallmatrix}\phantom{\sm}f\\\sm{}t\end{smallmatrix}\bigr](x),
$$ 
which shows that $[f\:\sm{}\hspace{-1.5pt}t]^{\scriptscriptstyle\operatorname{T}}(X)\subseteq Y\oplus T$ is closed.

\smallskip

\textquotedblleft{}(c)$\Longrightarrow$(d)\textquotedblright{} Let $s\in T$ be given with $q_T(s)=0$. By (c) this means $(0,s)\in\operatorname{ran}[f\:\sm\hspace{-1.5pt}t]^{\scriptscriptstyle\operatorname{T}}$. Thus there exists $x\in X$ with $f(x)=0$ and $-t(x)=s$. Since $f$ is injective we get $x=0$ and thus $s=-t(0)=0$.

\smallskip

\textquotedblleft{}(d)$\Longrightarrow$(b)\textquotedblright{} Let $(x_{\alpha})_{\alpha\in A}\subseteq X$ be a net such that $x_{\alpha}\rightarrow0$ in $(X,\tau_Y)$ and $t(x_{\alpha})\rightarrow s$ in $T$ hold. Through our identification of $X$ with $f(X)$ the former means $f(x_{\alpha})\rightarrow 0$ in $Y$. Consequently
$$
\bigl[\begin{smallmatrix}\phantom{\sm}f\\\sm{}t\end{smallmatrix}\bigr](x_{\alpha})=\bigl[\begin{smallmatrix}\phantom{\sm}f(x_{\alpha})\\\sm{}t(x_{\alpha})\end{smallmatrix}\bigr]\longrightarrow \bigl[\begin{smallmatrix}0\\s\end{smallmatrix}\bigr]
$$
and thus $(0,s)\in\overline{\operatorname{ran}[f\:\sm\hspace{-1.5pt}t]^{\scriptscriptstyle\operatorname{T}}}$. But then $q_T(s)=(0,s)+\overline{\operatorname{ran}[f\:\sm\hspace{-1.5pt}t]^{\scriptscriptstyle\operatorname{T}}}=0$ and $s=0$ by (d).

\smallskip

\textquotedblleft{}(b)$\Longrightarrow$(a)\textquotedblright{} Assume $(X,\tau_X)'\not=(X,\tau_Y)'$ and pick $\varphi\in(X,\tau_X)'\backslash(X,\tau_Y)'$. Since $\KK\in\Class$ and $\varphi\colon(X,\tau_X)\rightarrow\KK$ is linear and continuous, (b) implies that $\varphi\colon(X,\tau_Y)\rightarrow\KK$ has a closed graph. To get a contradiction, observe first that $\varphi$ cannot be the zero functional and that thus there exists $x\in X$ with $\varphi(x)=1$. Observe next that, as $\varphi$ is not $\tau_Y$-continuous, $\varphi^{-1}(0)\subseteq X$ is $\tau_Y$-dense. Therefore, we find a net $(x_{\alpha})_{\alpha\in A}\subseteq X$ such that $\varphi(x_{\alpha})=0$ for all $\alpha\in A$ and $x_{\alpha}\rightarrow x$ with respect to $\tau_Y$. But then
$$
(x_{\alpha},\varphi(x_{\alpha}))\rightarrow (x,0) \text{ in } (X,\tau_Y)\times\KK
$$
while $\varphi(x)\not=0$.
\end{proof}

\begin{lem}\label{LEM-PO-REG} Let $X\stackrel{f}{\rightarrow}Y\stackrel{g}{\rightarrow}Z$ be an $\LBr$-stable kernel-cokernel pair. Let $t\colon X\rightarrow T$ be an arbitrary morphism in $\LBr$. Then the space
$Q:=(Y\oplus T)/\overline{\operatorname{ran}[f\sm\hspace{-2pt}t]^{\scriptscriptstyle\operatorname{T}}}$ is regular.
\end{lem}
\begin{proof} By Propositions \ref{LBrKERCOK}(v) and \ref{LBr-SUR} the given pair $(f,g)$ is a kernel-cokernel pair in $\LB$. We form first the pushout of $f$ along $t$ in $\LB$, which is the $\LB$-cokernel of the map $[f\:\sm\hspace{-.5pt}t]^{\scriptscriptstyle\operatorname{T}}\colon X\rightarrow Y\oplus T$, i.e., precisely the space $Q$ in the lemma, together with the natural maps $q_Y$ and $q_T$. Next we use the pushout property in $\LB$ to obtain a map $c\colon (Y\oplus T)/\overline{\operatorname{ran}[f\sm\hspace{-2pt}t]^{\scriptscriptstyle\operatorname{T}}}\rightarrow Z$ with $c\circ q_T=0$ and $c\circ q_Y=g$. By Lemma \ref{AN-2a} we have $c=\cok q_T$ in $\LBr$. This yields the two upper squares of the following diagram:
\begin{equation}\label{DIAG}
\begin{tikzcd}
X\arrow{r}{f}\arrow{d}[swap]{t} & Y \arrow{d}{q_Y}\arrow{r}{g} &Z\arrow[equal]{d}{}\\[0.7em]
T \arrow{r}{q_T}\arrow[equal]{d}{} & \frac{Y\oplus T}{\overline{\operatorname{ran}[f\:\sm\hspace{0.5pt}t]^{\operatorname{T}}}}\arrow{d}{i}\arrow[dashed]{r}{c}&  Z\arrow[equal]{d}{}\\
T \arrow{r}[]{i\circ q_T}&\bigl(\frac{Y\oplus T}{\overline{\operatorname{ran}[f\:\sm\hspace{.5pt}t]^{\operatorname{T}}}}\bigr)^{\!\sharp}\arrow[dashed]{r}[]{c^{\sharp}} & Z.
\end{tikzcd}
\end{equation}
Next we form the pushout of $f$ along $t$ in $\LBr$, which is the cokernel of $[f\:\sm\hspace{-.5pt}t]^{\scriptscriptstyle\operatorname{T}}$ in $\LBr$ and thus by Proposition \ref{LBrKERCOK}(ii) the  triple $(Q^{\sharp},i\circ q_Y,i\circ q_T)$. Here, $i\colon Q\rightarrow Q^{\sharp}$ denotes the inclusion into the mackeyfication of $Q$. We thus obtain the left lower square in \eqref{DIAG} and observe that the left half of the diagram is a pushout in $\LBr$. Let finally  $c^{\sharp}$ be the mackeyfication of $c$. This makes the right lower square of \eqref{DIAG} commutative, which in turn implies  $c^{\sharp}\circ i\circ q_T =0$ and $ c^{\sharp}\circ i \circ q_Y=g$. This means that  $c^{\sharp}$ is precisely the map that arises from the left half of the diagram being an $\LBr$-pushout and consequently $c^{\sharp}=\cok(i\circ q_T)$ in $\LBr$, again by Lemma \ref{AN-2a}. Since $g$ is surjective by assumption, the same is true for $c$ and $c^{\sharp}$. 

\smallskip

By assumption, $f$ is $\LBr$-semistable. This implies that $i\circ q_T$ is a kernel in $\LBr$, while $c^{\sharp}$ is its $\LBr$-cokernel and is additionally surjective. The bottom row of \eqref{DIAG} is thus in particular an exact sequence of vector spaces. For the middle row the same holds: $c$ is surjective as noted above and $q_T$ is injective since $i\circ q_T$ is injective. As $c=\cok q_T$ holds in $\LB$, we have $c\circ q_T=0$. Finally, for $r\in c^{-1}(0)$ it follows $0=c(r)=c^{\sharp}(i(r))$. In view of the algebraic exactness of the bottom row, there exists $s\in T$ such that $(i\circ q_T)(s)=i(r)$. Since $i$ is injective, this implies $q_T(s)=r$ which establishes $q_T(T)=c^{-1}(0)$. Using the 5-Lemma in the category of vector spaces we get that $i$ is a bijection, whence an isomorphism of LB-spaces (either by the open mapping theorem or as the inclusion of a space into its mackeyfication is open onto its range).
\end{proof}

\begin{lem}\emph{[\hspace{1.5pt}joint work with J.~Wengenroth]}\label{LEM-W} Let $X\stackrel{f}{\rightarrow}Y\stackrel{g}{\rightarrow}Z$ be an algebraically exact sequence in $\LB$. Assume that $f$ is a weak isomorphism onto its range. If $X$ and $Z$ are regular, then so is $Y$.
\end{lem}

\begin{proof} Let $X=\ind_{n\in\NN}X_n$, $Y=\ind_{n\in\NN}Y_n$ and $Z=\ind_{n\in\NN}Z_n$ be given by sequences with inclusions as structure maps and form for each the standard resolution \eqref{RES}. As $X$ and $Z$ are regular, the quotient maps $\sigma_X$ and $\sigma_Z$ lift bounded sets.

\smallskip

Since $g$ is surjective, we have $\ind_{n\in\NN}g(Y_n)\cong Z$ where $g(Y_n)$ carries the Banach space topology from $Y_n/(g^{-1}(0)\cap Y_n)$ and thus we may w.l.o.g.\ assume that $g$ restricts to surjective maps $g_n\colon Y_n\rightarrow Z_n$ for all $n\in\NN$. Moreover, $\ind_{n\in\NN}g^{-1}(0)\cap Y_n\cong X$ from whence we may additionally assume w.l.o.g.\ that $f$ is the inclusion of a linear subspace and that $X_n=f^{-1}(0)\cap Y_n$ holds. By summing up the algebraically (and even topologically) exact sequences
$$
X_n\stackrel{f}{\longrightarrow}Y_n\stackrel{g}{\longrightarrow}Z_n
$$ 
of Banach spaces, we obtain the following commutative diagram, cf.~\cite[Prop 1.2]{Vogt92}, in which the top row is algebraically (and even topologically) exact\,---\,in particular the map $G$ below lifts bounded sets:
\begin{equation}\label{DIAG-VOGT}
\begin{tikzcd}
\Bigosum{n\in\mathbb{N}}{}X_n\arrow[r,"F"]\arrow[d,swap,"\sigma_X"] & \Bigosum{n\in\mathbb{N}}{}Y_n\arrow[r,"G"]\arrow[d,"\sigma_Y"] & \Bigosum{n\in\mathbb{N}}{}Z_n\arrow[d,"\sigma_Z"]\\
X\arrow[r,swap,"f"] & Y\arrow[r,swap,"g"] & Z.
\end{tikzcd}
\end{equation}
We show that $\sigma_Y$ lifts bounded sets: Let $B\subseteq Y$ be bounded. As $\sigma_X$ and $G$ lift bounded sets, there is a bounded set $C\subseteq\bigoplus_{n\in\NN}Y_n$ with $g(B)\subseteq(\sigma_Z\circ G)(C)$. Then $D:=B-\sigma_Y(C)$ is bounded in $Y$. Since in any lcs bounded and weakly bounded sets coincide, and since $f$ is a weak isomorphism onto its range, $f^{-1}(D)$ is weakly bounded and hence bounded in $X$. As $\sigma_X$ lifts bounded sets, there exists a bounded $E\subseteq\bigoplus_{n\in\NN}X_n$ with $f^{-1}(D)\subseteq\sigma_X(E)$. Then $F(E)$ and hence $K:=C+F(E)$ is bounded in $\bigoplus_{n\in\NN}Y_n$ and we claim that $B\subseteq \sigma_Y(K)$ holds. For $b\in B$ there is $c\in C$ with $g(b)=(\sigma_Z\circ G)(c)=(g\circ\sigma_Y)(c)$  so that
$$
b-\sigma_Y(c)\in g^{-1}(0)\cap D= f(X)\cap D,
$$
where we used the definition of $D$ and the algebraic exactness of the bottom row in \eqref{DIAG-VOGT}. There is thus $x\in X$ with $b-\sigma_Y(c)=f(x)$ and hence $x\in f^{-1}(D)\subseteq\sigma_X(E)$ so that there is $e\in E$ with $x=\sigma_X(e)$ and $f(\sigma_X(e))=b-\sigma_Y(c)$. Finally, $k:=c+F(e)\in K$ satisfies
$$
\sigma_Y(k)=\sigma_Y(c)+(\sigma_Y\circ F)(e)=\sigma_Y(c)+(f\circ \sigma_X)(e)=b.
$$
We have thus established that $\sigma_Y$ lifts bounded sets, which implies that $Y$ is regular.
\end{proof}

\begin{thm}\label{REG-EMB-EX} On $\LBr$ the following defines the maximal exact structure:
\begin{equation*}
\begin{aligned}
\Emaxreg&:=\bigl\{(f,g)\in\Callreg\:\big|\: (f,g)\;\text{is }\LBr\text{-stable}\hspace{1pt}\bigr\}\\
&\phantom{:}= \bigl\{X\stackrel{f}{\rightarrow}Y\stackrel{g}{\rightarrow}Z\in\LBr\:\big|\:(f,g) \text{ is algebraically exact, } f(X)\subseteq Y\text{ is well-located}\hspace{1.5pt}\bigr\}\\[1pt]
&\phantom{:}= \Emax\cap\LBr.
\end{aligned}
\end{equation*}
\end{thm}
\begin{proof}That $\Emaxreg$ as defined above is the maximal exact structure on $\LBr$ follows from Theorem \ref{PROP-MAX}. The last equality is true by definition. By Lemma \ref{LEM-W}, $\LBr\subseteq(\LB,\Emax)$ is extension closed, from whence it follows that $(\LBr,\Emax\cap\LBr)$ is exact, see e.g.\ \cite[Lem 10.20]{Buehler}. Consequently, $\Emax\cap\LBr\subseteq\Emaxreg$ must hold.

\smallskip

For the missing inclusion let $X\stackrel{\scriptscriptstyle f}{\rightarrow}Y\stackrel{\scriptscriptstyle g}{\rightarrow}Z$ in $\Emaxreg$ be given. Proposition \ref{LBr-SUR} implies that $g$ is surjective and hence we may use Proposition \ref{LBrKERCOK}(v) to conclude that $(f,g)$ is algebraically exact. By Lemma \ref{LEM-PO-REG} we get that for any $t\colon X\rightarrow Y$ the quotient $(Y\oplus T)/\overline{\operatorname{ran}[f\sm\hspace{-2pt}t]^{\scriptscriptstyle\operatorname{T}}}$ is regular. In particular, the $\LBr$-pushout of $f$ along any $\LBr$-morphism $t\colon X\rightarrow T$ coincides with the $\HDLCS$-pushout \eqref{LEM-D-PO}. As $f$ is an $\LBr$-kernel, Proposition \ref{LBrKERCOK}(iii) implies  that $f$ is injective and has a closed range; thus Lemma \ref{LEM-D}(ii) is applicable with $\Class=\LBr$. Since $f$ is an $\LBr$-semistable kernel, $q_T$ in \eqref{LEM-D-PO} is in particular injective. Using (d)$\Longrightarrow$(a) from Lemma \ref{LEM-D}(ii) we obtain that $f(X)\subseteq Y$ is well-located.
\end{proof}

We have $\Emaxreg\subseteq\Cregsur$ by definition; the latter indeed holds with a strict inclusion as the next example shows.

\begin{ex}\label{EX-LBr-2} Let $V=(v_n)_{n\in\mathbb{N}}$ be a decreasing sequence of strictly positive functions $v_n\colon\mathbb{N}\rightarrow\mathbb{R}$. We consider the K\"othe co-echelon space of order $1\leqslant p<\infty$, i.e.,
$$
k^{\hspace{1pt}p}(V):=\mathop{\ind}_{n\in\mathbb{N}}\,\ell^{\hspace{1pt}p}(v_n)\vspace{-3pt}
$$
with
$$
\ell^{\hspace{1pt}p}(v_n):=\bigl\{x\in\mathbb{K}^{\mathbb{N}}\:\big|\:\|x\|_{v_n}:=\Bigl(\Bigsum{j=1}{\infty}v_n(j)|x(j)|^p\Bigr)^{1/p}<\infty\bigr\},
$$
which is, for any $V$, a complete, hence regular inductive limit, see, e.g., Bierstedt, Meise, Summers \cite[Thm 2.3(a)]{BMS}. We now take a sequence $V$ that is not `regularly decreasing' as defined in \cite[Dfn 3.1 with $I=\NN$]{BMS}; concrete examples can be found in \cite[Ex 4.11(2) and (3)]{BMS}. In Vogt's \cite[Section 5]{Vogt92} (LF-)notation we get $E^{\hspace{1pt}1}=k^{\hspace{1pt}1}(V)$ by putting $a_{j;\hspace{0.5pt}n,k}:=v_n(j)$. His condition (WQ) is then always satisfied, as can be checked directly or deduced from \cite[Thms 5.10 and 5.14]{Vogt92} in view of the completeness of $k^{\hspace{1pt}1}(V)$. As noted by Bierstedt, Bonet \cite[Lem 3.2 and the sentence right before it]{BB94}, Vogt's condition (Q) is equivalent to $V=(v_n)_{n\in\NN}$ being regularly decreasing. As we considered a sequence $V$ which does not enjoy the latter property, \cite[Thm 5.6]{Vogt92} implies that $k^{\hspace{1pt}1}(V)$ is not `weakly acyclic' \cite[p.~58]{Vogt92} which means by definition that in the canonical resolution 
\begin{equation*}
\Bigosum{n\in\mathbb{N}}{}\ell^{\hspace{0.5pt}1}(v_n)\stackrel{\hspace{-2pt}d}{\longrightarrow}\Bigosum{n\in\mathbb{N}}{}\ell^{\hspace{0.5pt}1}(v_n)\stackrel{\sigma}{\longrightarrow}k^{\hspace{0.5pt}1}(V)
\end{equation*}
the map $d$ is not a weak isomorphism onto its range. Consequently, we obtain $(d,\sigma)\in\Cregsur\backslash\Emaxreg$.\diam{}
\end{ex}

Let us remark, that the notion of (weak) acyclicity was initially introduced in the functional analytic context by Palamodov \cite[p.~33]{Pala71}, who called an inductive spectrum of lcs acyclic [weakly acyclic], if in the corresponding canonical resolution (as in Remark \ref{Stand-RES} but with possibly infinite directed index set) the map $d$ is an isomorphism [weak isomorphism] onto its range. Vogt showed that in the case of an LF-space both properties are invariant under the choice of the spectrum and only depend on the limit space. 

\begin{thm}\label{ETOP-REG} On $\LBr$ the class $\Etopreg=\Etop\cap\LBr$ of topologically exact sequences is an exact structure. 
\end{thm}
\begin{proof} By Lemma \ref{LEM-W} the category $\LBr\subseteq(\LB,\Etop)$ is extension closed, hence $\Etopreg=\Etop\cap\LBr$ is an exact structure by \cite[Lem 10.20]{Buehler}.
\end{proof}

\begin{ex}\label{EX-Vogt-2} Let $V=(v_n)_{n\in\NN}$, precisely as in Example \ref{EX-LBr-2}, be a decreasing sequence that is not regularly decreasing. As elaborated in the aforementioned example, we read $k^{\hspace{1pt}p}(V)$ as an LF-space and apply \cite[Thms 5.7 and 5.10]{Vogt92}, now for $1<p<\infty$, to obtain a kernel-cokernel pair $(d,\upsigma)$ in $\LBr$, in which $\sigma$ is surjective and $d$ is a weak isomorphism, but not an isomorphism, onto its range. Hence $(d,\sigma)\in\Emaxreg\backslash\Etopreg$.\diam{}
\end{ex}

\vspace{3pt}

\section{Complete LB-spaces}\label{SEC-COM}

Now we consider the category $\LBc$ of complete LB-spaces, which is a full additive subcategory of $\LB$. Before we look at kernels and cokernels let us mention the following three key obstacles that we have to navigate around:\vspace{4pt}

\begin{compactitem}

\item[1.] It is unknown if the mackeyfication $X^{\sharp}$ of an LB-space is always complete.

\vspace{3pt}

\item[2.] It is unknown if the completion $\widehat{X}$ of an LB-space $X$ is always an LB-space.

\vspace{3pt}

\item[3.] It is unknown if for a closed subspace $Y\subseteq X$ of a complete  LB-space the space $Y^{\flat}$ is always complete.

\vspace{4pt}

\end{compactitem}

If the answer to Question \ref{Q-G} were to be `yes', then all three statements would hold, whereas if it were `no', then there would exist an LB-space whose mackeyfication is a proper subspace of its completion. It is well-known that\,---\,outside the class of LB-spaces\,---\,there exist lcs which are Mackey complete but not even sequentially complete, see \cite[Ex 5.1.12]{BPC}. Furthermore, there exist bornological spaces with a non-bornological completion, see \cite[Section 6.6]{BPC}. On the other hand, the so-called Moscatelli type LB-spaces constitute a whole class of LB-spaces which are complete iff they are regular and whose completions are always LB-spaces, see Bonet, Dierolf, Ku\ss{} \cite{BD89, DK}. Finally, if there were to exist an LB-space whose completion is not LB, then Question \ref{Q-G} could be answered with `no'.

\begin{dfn}\label{DFN-LB-COMP} Let $X$ be an LB-space. An injective map with dense range $i\colon X\rightarrow Y$ into a complete LB-space is called an \emph{LB-completion} of $X$, if for any $f\colon X\rightarrow Z$ into a complete LB-space there exists a unique map $\widehat{f}^{\hspace{3pt}\text{\raisebox{0.85mm}{$\scriptscriptstyle\LB$}}}\colon Y\rightarrow Z$ with $\widehat{f}^{\hspace{3pt}\text{\raisebox{0.85mm}{$\scriptscriptstyle\LB$}}}\circ i=f$. If an LB-completion exists, then it is unique up to unique isomorphism and we denote it by $\widehat{X}^{\text{\raisebox{0.85mm}{$\scriptscriptstyle\LB$}}}$.\diam{}
\end{dfn}

 If for an LB-space $X$ the (usual) completion $\widehat{X}$ is an LB-space, then $\widehat{X}=\widehat{X}^{\text{\raisebox{0.85mm}{$\scriptscriptstyle\LB$}}}$ and $i$ will be open onto its range. We do not know if there is an LB-space for which an LB-completion exists such that $i$ is not open onto its range. 

\begin{prop}\label{LBcKERCOK}We consider the category $\LBc$. \vspace{3pt}
\begin{compactitem}
\item[(i)] A map in $\LBc$ is monic/epic iff it is injective/has a dense range.\vspace{4pt}

\item[(ii)] A map $f\colon X\rightarrow Y$ has a kernel in $\LBc$ iff $f^{-1}(0)^{\flat}$ is complete. If this is the case, then the kernel of $f$ is the inclusion $f^{-1}(0)^{\flat}\rightarrow X$.

\vspace{4pt}

\item[(iii)] A map $f\colon X\rightarrow Y$ has a cokernel iff $Y/\overline{f(X)}$ has an LB-completion. If this is the case, then the cokernel is $i\circ q\colon X\rightarrow(Y/\overline{f(X)})^{\raisebox{2pt}{$\scriptstyle\wedge_{\LB}$}}$.

\vspace{0.5pt}

\item[(iv)] A sequence $X\stackrel{f}{\rightarrow}Y\stackrel{g}{\rightarrow}Z$ in $\LBc$ is a kernel-cokernel pair in $\LBc$ with $g$ being a surjection iff it is a kernel-cokernel pair in $\LB$.
\end{compactitem}
\end{prop}

\begin{proof}(i) This follows from $\KK\in\LBc$ by using the map $g\colon\KK\rightarrow X$, $\lambda\mapsto\lambda x$, respectively the functional $g\colon Y\rightarrow\KK$ with $g(y_0)=1$ for some $y_0\in Y\backslash\overline{f(X)}$ but $g|_{\overline{f(X)}}\equiv0$. 

\smallskip

(ii) Assume that $f$ has a kernel $k\colon\ker f\rightarrow X$. Then $k$ is monic and thus by (i) injective. We may thus assume that $\ker f\subseteq X$ is a linear subspace and that $k$ is the inclusion map. Since $f\circ k=0$, it follows $\ker f\subseteq f^{-1}(0)$ as sets and by looking at the map $\mathbb{K}\rightarrow X$, $\lambda\mapsto \lambda x$ for $x\in f^{-1}(0)$ we deduce $\ker f=f^{-1}(0)$ as linear spaces. Since $k$ is continuous and the inclusion of a linear subspace, we get that the identity $\ker f\rightarrow (f^{-1}(0),\uptau_X)$ is continuous. Taking the $\flat$-topologies on both sides preserves the continuity of the identity, but does not change the topology on $\ker f$, as the latter was already an LB-space. By the open mapping theorem $\ker f\rightarrow (f^{-1}(0),\uptau_X)^{\flat}$ is an isomorphism. The other direction follows from Section \ref{SEC-2}.

\smallskip

(iii) If $(Y/\overline{f(X)})^{\raisebox{2pt}{$\scriptstyle\wedge_{\LB}$}}$ exists, then it is straightforward to verify the universal property. For the other direction let $c\colon Y\rightarrow C$ be an $\LBc$-cokernel of $f$. This means by definition that $c\circ f=0$ and thus $\overline{f(X)}\subseteq c^{-1}(0)$. Let $y\in Y\backslash\overline{f(X)}$. Then by Hahn-Banach there exists $\varphi\colon Y\rightarrow\mathbb{K}$ with $\varphi(\overline{f(X)})=0$ and $\varphi(y)=1$. The first condition implies $\varphi\circ f=0$ and thus by the cokernel property in $\LBc$ we get $h\colon C\rightarrow\KK$ satisfying $(h\circ c)(y)=1$. In particular, $c(y)\not=0$ from whence it follows $c^{-1}(0)=\overline{f(X)}$. Let now $q\colon Y\rightarrow Y/\overline{f(X)}$ be the quotient map. Then $q=\cok f$ in $\LB$ and we get
$$
i\colon Y/\overline{f(X)}\rightarrow C\;\text{ with }\;i\circ q = c
$$ 
and $i$ is injective since $i(y+\overline{f(X)})=0$ implies $0=(i\circ q)(y)=c(y)$. By what we noted above it follows $y\in\overline{f(X)}$. Since $c$, as an $\LBc$-cokernel, is epic, we get from (i) that $i$ has a dense range. Let finally $g\colon Y/\overline{f(X)}\rightarrow Z$ be given with $Z\in\LBc$. Then $g\circ q\circ f=0$ and as $i\circ q=c =\cok f$ in $\LBc$, we get a unique map $h\colon C\rightarrow Z$ with $h\circ i\circ q = g\circ q$:
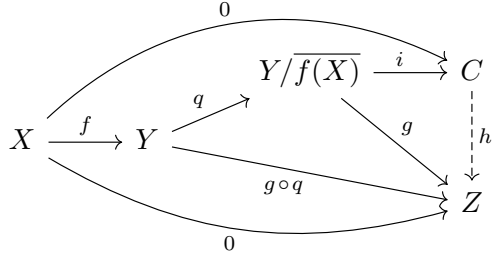
\begin{equation*}
\begin{tikzcd}[column sep=28pt,row sep=13pt]
 & & Y/\overline{f(X)}\arrow{ddr}{g}\arrow{r}[]{\hspace{-8pt}i} & C \arrow[dashed]{dd}{h} \\[-5pt]
  X\arrow{r}{f}\arrow[bend left =35]{rrru}{0}\arrow[bend right =23]{rrrd}[swap]{0} & Y\arrow{ur}{q}\arrow{drr}[swap]{g\hspace{1pt}\circ\hspace{1pt}q} &\\[-5pt]
  &  & & Z.
\end{tikzcd}
\end{equation*}
Since $q$ is surjective, this implies $h\circ i = g$, and since $i$ has a dense range, there can only be one such map $h$. Thus, we have established that $i\colon Y/\overline{f(X)}\rightarrow C$ is the LB-completion of $Y/\overline{f(X)}$.

\smallskip

(iv) \textquotedblleft{}$\Longrightarrow$\textquotedblright{} By assumption $f=\ker g$, and thus we get from (ii) that $f$ is injective and has a closed range. By (iii) we know that $Z$ is the LB-completion of $Y/f(X)$, or, more precisely, that $g=i \circ q$ where $q\colon Y\rightarrow Y/f(X)$ is the quotient map and $i\colon Y/f(X)\rightarrow Z$ is injective and has a dense range. However, as $g$ is by assumption surjective, $i$ has to be surjective, and is thus an isomorphism by the open mapping theorem. Consequently, $g^{-1}(0)=f(X)$ and $(f,g)$ is a kernel-cokernel pair in $\LB$ by Proposition \ref{LB-PROP}(v). 

\smallskip

\textquotedblleft{}$\Longleftarrow$\textquotedblright{} If $X\stackrel{\scriptscriptstyle f}{\rightarrow}Y\stackrel{\scriptscriptstyle g}{\rightarrow}Z$ is a kernel-cokernel pair in $\LB$ with $X,Y,Z\in\LBc$, then $f=\ker g$ and $g=\cok f$ hold in $\LBc\subset\LB$, too.
\end{proof}

We denote by $\Callcom$ the class of all kernel-cokernel pairs in $\LBc$. There exist $(f,g)\in\Callcom$ with non-surjective $g$: Let $X$ be a non-regular LB-space whose completion is an LB-space, cf.~\cite[Ex 8.8.8 and 8.8.9]{BPC} or \cite{DK} for concrete examples, and consider its standard resolution $(d,\sigma)$ as in Remark \ref{Stand-RES}. If $i\colon X\rightarrow\widehat{X}$ denotes the inclusion, then $(d,i\circ\sigma)\in\Callcom$ and $i\circ\sigma$ is not surjective.
 
\smallskip

Even though our information about kernels and cokernels is limited compared to $\LBr$, we can consider on $\LBc$ three natural conflation structures. We start this time with the topologically exact sequences.

\begin{thm} On $\LBc$ the topologically exact sequences $\Etopcom:=\Etop\cap\LBc$ form an exact structure.
\end{thm}
\begin{proof} By \cite[Prop 1.3]{DR81}, see also \cite[Chapitr\'e III, \S\hspace{0.5pt}3, Ex 9]{Bourbaki}, \cite[Prop 2.4.2]{BPC} and \cite[Table on p.~2114]{DS12}, the following holds: If  $X\rightarrow Y\rightarrow Z$ is a topologically exact sequence of lcs with complete $X$ and $Z$, then $Y$ is complete as well. We thus get that $\LBc\subseteq(\LB,\Etop)$ is extension closed, hence $\Etopcom$ is an exact structure by \cite[Lem 10.20]{Buehler}.
\end{proof}

\begin{prop}\label{LBcKaroubi} The category $\LBc$ is karoubian.
\end{prop}
\begin{proof}Let $p\colon X\rightarrow X$ be an idempotent in $\LBc$. Then $(p(X),\uptau_X)\subseteq X$ is a direct summand and \cite[Cor 6.2.2]{BPC} yields that $(p^{-1}(0),\uptau_X)$ is ultrabornological and thus the latter coincides with the $\flat$-topology. Secondly, $p^{-1}(0)\subseteq X$ is closed and therefore complete. Consequently, the inclusion $p^{-1}(0)^{\flat}\rightarrow X$ is a kernel of $p$.
\end{proof}

\begin{lem}\label{LEM-StaPa}\begin{myitemize}\setlength{\itemindent}{-10pt}\item[(i)] In $\LBc$ every semistable cokernel is surjective.

\setlength{\itemindent}{0pt}\vspace{1pt}

\item[(ii)] Let $X\stackrel{f}{\rightarrow}Y\stackrel{g}{\rightarrow}Z$ be an $\LBc$-stable kernel-cokernel pair and let $t\colon X\rightarrow T$ be an arbitrary morphism in $\LBc$. Then $Q:=(Y\oplus T)/\overline{\operatorname{ran}[f\sm\hspace{-2pt}t]^{\scriptscriptstyle\operatorname{T}}}$ is complete.
\end{myitemize}
\end{lem}
\begin{proof} (i) It is enough to repeat the proof of \textquotedblleft{}$\Longrightarrow$\textquotedblright{} in Proposition \ref{LBr-SUR} and while doing so to observe that, although pullbacks do not exist in general in $\LBc$, this specific pullback $(P,p_{\KK},p_Y)$ exists by assumption, cf.~Definition \ref{AN-3}.

\smallskip

(ii) To see this, we repeat the proof of Lemma \ref{LEM-PO-REG}, replacing the mackeyfication with the LB-completion and observing that, at the very end of the proof, the map $i\colon Q\rightarrow\widehat{Q}^{\text{\raisebox{0.85mm}{$\scriptscriptstyle\LB$}}}$ is an isomorphism due to the open mapping theorem. Again, the LB-completion exists by Definition \ref{AN-3} in combination with Proposition \ref{LBcKERCOK}(iii).
\end{proof}

\begin{thm}\label{LBc-MAX-THM} On $\LBc$ the following defines the maximal exact structure:
\begin{equation*}
\begin{aligned}
\Emaxcom&:=\bigl\{(f,g)\in\Callcom\:\big|\: (f,g)\;\text{is }\LBc\text{-stable}\hspace{1pt}\bigr\}\\
&\phantom{:}= \bigl\{X\stackrel{f}{\rightarrow}Y\stackrel{g}{\rightarrow}Z \in\LBc\:\big|\:(f,g)\,\text{is algebraically exact,}\,f(X)\subseteq Y\\
&\hspace{32pt}\text{is well-located, }\forall\,t\colon X\rightarrow T \text{ in } \LBc\colon\textstyle\frac{Y\oplus T}{\overline{\operatorname{ran}[f\:\sm\hspace{-0.5pt}t]^{\operatorname{T}}}} \text{ is complete}\hspace{1pt}\bigr\}.
\end{aligned}
\end{equation*}
\end{thm}
\begin{proof}By Proposition \ref{LBcKaroubi} and Theorem \ref{PROP-MAX} the above defined class $\Emaxcom$ is the maximal exact structure on $\LBc$. We have to show the equality.

\smallskip

\textquotedblleft{}$\subseteq$\textquotedblright{} In view of Proposition \ref{LBcKERCOK} and Lemma \ref{LEM-StaPa}(ii) we can show this ana\-logously to the corresponding inclusion in Theorem \ref{REG-EMB-EX}.

\smallskip

\textquotedblleft{}$\supseteq$\textquotedblright{}  Let $X\stackrel{\scriptscriptstyle f}{\rightarrow}Y\stackrel{\scriptscriptstyle g}{\rightarrow}Z$ be given with the properties stated. Then $(f,g)$ is a kernel-cokernel pair in $\LB$ and thus in $\LBc$ by Proposition \ref{LBcKERCOK}(iv). Given $t\colon X\rightarrow T$ in $\LBc$, the third property in the second class mentioned in the theorem guarantees that the $\LBc$-pushout of $f$ along $t$ exists and coincides with the $\LB$-pushout of $f$ along $t$. By applying Lemma \ref{AN-2a} in $\LB$, we get the following diagram, in which $c$ is the $\LB$-cokernel of $q_T$:
\begin{equation}\label{NEW}
\begin{tikzcd}
X\arrow{r}{f}\arrow{d}[swap]{t}\commutes[\mathrm{PO}_{\LB}]{dr} & Y \arrow{d}{q_Y}\arrow{r}{g}&Z\arrow[equals]{d}\\[4pt]
T \arrow{r}[swap]{q_T} & \frac{Y\oplus T}{\overline{\operatorname{ran}[f\:\sm\hspace{.5pt}t]^{\scriptscriptstyle\operatorname{T}}}}\arrow[dashed,swap]{r}{c} & Z.
\end{tikzcd}
\end{equation}
As $f(X)\subseteq Y$ is well-located by assumption, Remark \ref{RMK-LB-LAST}(ii) yields that $f$ is an $\LB$-semistable kernel. Consequently $q_T$ is a kernel in $\LB$ and thus $(q_T,c)\in\Call$ holds. As in the last row of \eqref{NEW} all spaces are complete and $j$ is surjective, Proposition \ref{LBcKERCOK}(iv) shows $(q_T,c)\in\Callcom$ meaning that $q_T$ is an $\LBc$-kernel. We thus have established that $f$ is an $\LBc$-semistable kernel.

\smallskip

Notice that we now already know that $(f,g)$ belongs to $\Emax$, as in $\LB$ all cokernels are semistable. It remains to show that $g$ is semistable in $\LBc$ as well. For this let $t\colon T\rightarrow Z$ in $\LBc$ be given. We claim that the space in the left upper corner of the $\LB$-pullback square
\begin{equation}\label{LB-PB-NEU}
\begin{tikzcd}
 {[t\:\sm{}\hspace{-3pt}g]^{-1}}(0)^{\flat}\arrow{r}{p_T}\arrow[swap]{d}{p_Y}\commutes[\mathrm{PB}_{\LB}]{dr} & T\arrow{d}{t} \\[0.6em]
 Y\arrow[]{r}[swap]{g} & Z
\end{tikzcd}
\end{equation}
is complete. For this purpose let $E_0\hookrightarrow E_1\hookrightarrow\cdots$ be a defining sequence of Banach spaces for the LB-space $T\oplus Y$ and let $(P,\tau_P):={[t\:\sm{}\hspace{-3pt}g]^{-1}}(0)\subseteq T\oplus Y$ carry the subspace topology\,---\,in which it is complete. We have to show that $P^{\hspace{1pt}\flat}=\mathop{\ind}_{n\in\NN} P\cap E_n$ is complete, too. We form the $\LB$-kernel of $p_T\colon P^{\hspace{1pt}\flat}\rightarrow T$. This is the inclusion map
$$
i\colon K^{\flat}=\mathop{\ind}_{n\in\NN}p_T^{-1}(0)\cap(P\cap E_n)=\mathop{\ind}_{n\in\NN}p_T^{-1}(0)\cap E_n \longrightarrow P^{\hspace{1pt}\flat}.
$$
As \eqref{LB-PB-NEU} is an $\LB$-pullback and $f=\ker g$, we get an isomorphism
\begin{equation}\label{NEW-2}
\begin{tikzcd}
K^{\flat}\arrow{r}{i}\arrow[swap]{d}{\sim} &P^{\hspace{1pt}\flat}\arrow{r}{p_T}\arrow[swap]{d}{p_Y}\commutes[\mathrm{PB}_{\LB}]{dr} & T\arrow{d}{t} \\[0.6em]
X\arrow[swap]{r}{f}& Y\arrow[]{r}[swap]{g} & Z
\end{tikzcd}
\end{equation}
as parallel maps in a pullback have isomorphic kernels. Using this and that $(f,g)\in\Emax$ holds, we obtain $(K^{\flat})'=X'=f(X)'$, where $f(X)$ carries the subspace topology of $Y$. Moreover, it holds $(i,p_T)\in\Emax$, and thus $(K^{\flat})'=i(K^{\flat})'=(K,\tau_{P^{\hspace{0.5pt}\flat}|K})'$, where $K=p_T^{-1}(0)$ carries the subspace topology of $P^{\hspace{1pt}\flat}$. Now we observe that the pullback of $g$ along $t$ in $\HDLCS$ is given by the right square in the next diagram, where $p_T$ and $p_Y$ are the same maps as above, but their domains carry now the subspace topology of $P=T\oplus Y$. We thus get another isomorphism of kernels, but now taken in $\HDLCS$, i.e.,
\begin{equation*}
\begin{tikzcd}
K\arrow{r}{i}\arrow[swap]{d}{\sim} &P\arrow{r}{p_T}\arrow[swap]{d}{p_Y}\commutes[\mathrm{PB}_{\HDLCS}]{dr} & T\arrow{d}{t} \\[0.6em]
f(X)\arrow[swap]{r}{j}& Y\arrow[]{r}[swap]{g} & Z
\end{tikzcd}
\end{equation*}
where $f(X)$ carries the subspace topology of $Y$, $K=p_T^{-1}(0)$ carries the subspace topology of $P$ and $j$ denotes the inclusion map. In particular it follows $f(X)'=(K,\tau_{P|K})'$. Putting everything together we see that all three topologies on $K$ lead to the same dual space, which we denote by
$$
K':=(K,\tau_{P^{\hspace{0.5pt}\flat}|K})'=i(K^{\flat})'=(K^{\flat})'=X'=f(X)'=(K,\tau_{P|K})'.
$$
We now look at the dual pair $(K,K')$, for which 
$$
\scalebox{0.98}{$\tau(K,K')=\tau(K,(K,\tau_{P^{\hspace{0.5pt}\flat}|K})')\supseteq\tau_{P^{\hspace{0.5pt}\flat}|K}\supseteq \tau_{P|K} \supseteq\sigma(K,(K,\tau_{P|K})')=\sigma(K,K')$}
$$
holds and thus both $\tau_{P^{\hspace{0.5pt}\flat}|K}$ and $\tau_{P|K}$ are $(K,K')$-admissible by \cite[Mackey-Arens thm 23.8]{MV}. The two topologies thus satisfy Robertson's closed neighborhood condition, that is $\tau_{P^{\hspace{0.5pt}\flat}|K}$ has a basis of $0$-neighborhoods consisting of $\tau_{P|K}$-closed sets; use e.g.~\cite[Rmk 23.9]{MV}. It follows that $(K,\tau_{P^{\flat}|K})$ is complete; use e.g.~\cite[Lem 23.13]{MV}. We thus obtain that
$$
(K,\tau_{P^{\flat}|K})\stackrel{i}{\longrightarrow}P^{\hspace{1pt}\flat}\stackrel{p_T}{\longrightarrow}T
$$
is a topologically exact sequence in $\HDLCS$ with $(K,\tau_{P^{\flat}|K})$ and $T$ being complete. By what we noted in the proof of Proposition \ref{LBcKaroubi}(ii) it follows that $P^{\hspace{1pt}\flat}$ is complete, too. Therefore we have established that the space $P^{\hspace{1pt}\flat}={[t\:\sm{}\hspace{-3pt}g]^{-1}}(0)^{\flat}$ in the left upper corner of \eqref{LB-PB-NEU} is a complete LB-space. It follows that \eqref{NEW-2} actually produces a kernel-cokernel pair $(i,p_T)$ in $\LBc$ and $g$ is in particular an $\LBc$-semistable cokernel. 
\end{proof}

\begin{thm}\label{LBc-Defl} On $\LBc$ the maximal (strongly) deflation-exact structure satisfies:
$$
\begin{aligned}
\Dcom&:=\bigl\{(f,g)\in\Callcom\:\big|\: g \text{ is } \LBc\text{-semistable}\hspace{1pt}\bigr\}\\
&\phantom{:}\subseteq\Cregsur\cap\LBc.
\end{aligned}
$$
\end{thm}
\begin{proof} That $\Dcom$ as defined is the maximal deflation-exact structure on $\LBc$ and that it is strongly deflation exact  follows from Theorem \ref{PROP-MAX-D}. Given $(f,g)\in\Dcom$, the map $g$ is surjective by Lemma \ref{LEM-StaPa}(i). Proposition \ref{LBcKERCOK}(iv) yields that $(f,g)$ is a kernel-cokernel pair in $\LB$. The inclusion follows then from Theorem \ref{C-REG-SUR}.
\end{proof}

\begin{rmk} We conclude with the following remarks.\vspace{3pt}

\begin{compactitem}

\item[(i)] We have $\Emaxcom\subseteq\Emaxreg\cap\LBc$ and equality is equivalent to the condition in \cite[Dfn 2.4(i)]{DS12}; it is unknown if the latter holds.

\vspace{3pt}

\item[(ii)] It is unknown if in Theorem \ref{LBc-Defl} equality holds, that is, if in the description of $\Dcom$ one may replace `semistable' with `surjective'.

\vspace{3pt}

\item[(iii)] We have $\Etopcom\subset\Emaxcom\subseteq\Dcom$. That the first inclusion is strict follows from Example \ref{EX-Vogt-2}; if the second one is strict is unknown.
\diam{}
\end{compactitem}
\end{rmk}

\vspace{-3pt}

\section{Exact functors}\label{SEC-EX-FKT}

In the previous sections we have defined conflation structures on the categories of all/regular/\linebreak{}complete LB-spaces. This leads to the following diagram of exact functors. All horizontal arrows are the identity, all properties of subcategories/functors that we know are indicated in the diagram.
\begin{equation*}
\begin{tikzcd}[
hookarrow/.style={{Hooks[left]}->},
hookarrowr/.style={{Hooks[right]}->}, column sep =3.1em, row sep = 2.5em]
(\LB,\Etop)\arrow{r}{\stackrel{\stackrel{\text{doesn't}}{\text{reflect}}}{\scriptscriptstyle\text{exactness}}}& (\LB,\Emax)\arrow{r}{\stackrel{\stackrel{\text{doesn't}}{\text{reflect}}}{\scriptscriptstyle\text{exactness}}}  & (\LB,\Dmax)\\[9pt]
(\LBr,\Etopreg)\arrow{r}[swap]{\stackrel{\stackrel{\text{doesn't}}{\text{reflect}}}{\scriptscriptstyle\text{exactness}}}\arrow[hookarrow]{u}{\stackrel{\stackrel{\text{ext'closed}}{\text{and reflects}}}{\scriptscriptstyle\text{exactness}}} &  (\LBr,\Emaxreg)\arrow[hookarrow]{u}{\stackrel{\stackrel{\text{ext'closed}}{\text{and reflects}}}{\scriptscriptstyle\text{exactness}}}\arrow{r}[swap]{\stackrel{\stackrel{\text{doesn't}}{\text{reflect}}}{\scriptscriptstyle\text{exactness}}}  &(\LBr,\Cregsur)\arrow[hookarrow,swap]{u}{\stackrel{\text{reflects}}{\scriptscriptstyle\text{exactness}}}\\[9pt]
(\LBc,\Etopcom)\arrow[hookarrow]{u}{\stackrel{\stackrel{\text{ext'closed}}{\text{and reflects}}}{\scriptscriptstyle\text{exactness}}}\arrow{r}[swap]{\stackrel{\stackrel{\text{doesn't}}{\text{reflect}}}{\scriptscriptstyle\text{exactness}}}& (\LBc,\Emaxcom)\arrow[hookarrow,swap]{u}{}\arrow[swap]{r}{}  & (\LBc,\Dcom)\arrow[hookarrow,swap]{u}{}
\end{tikzcd}
\end{equation*}


\vspace{3pt}
\section{Acyclic complexes}\label{SEC-AC}

For an additive category $\mathcal{A}$ we denote by $\C(\mathcal{A})$ the category of cochain complexes. The following definition is standard, see \cite[Dfn 7.1]{BC13}, \cite[Dfn 2.14]{HKRW} for conflation categories and, e.g., \cite[Dfn 10.1]{Buehler} for exact categories.

\begin{dfn} Let $(\mathcal{A},\CC)$ be a conflation category. We say that a cochain complex $X^{\bullet}=(X^n,d^{\hspace{1pt}n})_{n\in\ZZ}$ is \emph{$\CC$-acyclic in degree $n$} if $d^{\hspace{1pt}n-1}$ factors as
\begin{equation*}
\begin{tikzcd}[column sep=1.5em, row sep=1.7em]
X^{n-1}\arrow{rr}{d^{\hspace{0.5pt}n-1}}\arrow{rd}[swap]{p^{\hspace{0.5pt}n-1}}& & X^n \\
&Z^{\hspace{1pt}n}\arrow{ru}[swap]{i^{\hspace{0.5pt}n-1}} & 
\end{tikzcd}
\end{equation*}
where $i^{\hspace{0.5pt}n-1}$ is an inflation and a kernel of $d^{\hspace{1pt}n}$ and $p^{\hspace{0.5pt}n-1}$ is a deflation and a cokernel of $d^{\hspace{1pt}n-2}$. The complex $X^{\bullet}$ is \emph{$\CC$-acyclic} if it is $\CC$-acyclic in every degree and we denote by $\Ac(\mathcal{A},\CC)\subseteq\C(\mathcal{A})$ the full subcategory of $\CC$-acyclic complexes.\diam{}
\end{dfn}

Below we first characterize the acyclic complexes of LB-spaces w.r.t.\ the three conflation structures that we considered in Section \ref{SEC-2a}. Then we will treat the subcategories of regular resp.~complete LB-spaces.

\begin{prop}\label{LB-ACYC} Let $X^{\bullet}\in\C(\LB)$. Then $X^{\bullet}$ is\vspace{2pt}
\begin{compactitem}

\item[(i)] $\Dmax$-acyclic iff $\;\forall\hspace{1.8pt}n\in\mathbb{Z}\colon \ran d^{\hspace{1pt}n-1}=(d^{\hspace{1pt}n})^{-1}(0)$ algebraically,\vspace{4pt}

\item[(ii)] $\Emax$-acyclic iff $\;\forall\hspace{1.8pt}n\in\mathbb{Z}\colon(\ran d^{\hspace{1pt}n-1})'=(d^{\hspace{1pt}n})^{-1}(0)^{\flat})'$ algebraically,\vspace{4pt}

\item[(iii)] $\Etop$-acyclic iff $\;\forall\hspace{1.5pt}n\in\mathbb{Z}\colon \ran d^{\hspace{1pt}n-1}=(d^{\hspace{1pt}n})^{-1}(0)^{\flat}$ topologically.
\end{compactitem}
\end{prop}
\begin{proof}\textcircled{1} If $X^{\bullet}$ is $\Dmax$-acyclic in degree $n$, then $d^{\hspace{1pt}n-1}$ is \emph{strict}, meaning that the parallel of $d^{\hspace{1pt}n-1}$ is an isomorphism; indeed $p^{\hspace{0.5pt}n-1}$ is the coimage and $i^{\hspace{0.5pt}n-1}$ is the image of $d^{\hspace{1pt}n-1}$. Consequently, the image of $d^{\hspace{1pt}n-1}$ is the inclusion $d^{\hspace{1pt}n-1}(X^{n-1})^{\flat}\hookrightarrow X^n$. On the other hand, the map $(d^{\hspace{1pt}n})^{-1}(0)^{\flat}\hookrightarrow X^n$ is also an image of $d^{\hspace{1pt}n-1}$. The inclusion map $d^{\hspace{1pt}n-1}(X^{\hspace{1pt}n-1})^{\flat}\rightarrow(d^{\hspace{1pt}n})^{-1}(0)^{\flat}$ must thus be an isomorphism and we get in particular the equality of sets $d^{\hspace{1pt}n-1}(X^{\hspace{0.5pt}n-1})=(d^{\hspace{1pt}n})^{-1}(0)$. This shows \textquotedblleft{}$\Longrightarrow$\textquotedblright{} in (i) and for (ii) and (iii) it is enough to observe that $i^{\hspace{0.5pt}n-1}$ being an $\Emax$- or $\Etop$-inflation implies that $\ran i^{\hspace{0.5pt}n-1}=(d^{\hspace{1pt}n})^{-1}(0)\subseteq X^n$ is well-located, respectively a limit subspace, which yields the corresponding equality.  

\medskip

\textcircled{2} For \textquotedblleft{}$\Longleftarrow$\textquotedblright{} we see that $i^{\hspace{0.5pt}n-1}$, as a kernel, is always a $\Dmax$-inflation. We have to show that $p^{\hspace{0.5pt}n-1}\colon X^{n-1}\rightarrow\ker d^{\hspace{1pt}n}=(d^{\hspace{1pt}n})^{-1}(0)^{\flat}$, $x\mapsto d^{\hspace{1pt}n-1}(x)$ is the cokernel of $d^{\hspace{1pt}n-2}$. By construction $p^{\hspace{0.5pt}n-1}\circ d^{\hspace{1pt}n-2}=0$. Let $c\colon X^{\hspace{1pt}n-1}\rightarrow C$ be such that $c\circ d^{\hspace{1pt}n-2}=0$. Observe that the algebraic equality $\ran d^{\hspace{1pt}n-1}=(d^{\hspace{1pt}n})^{-1}(0)$ implies that $d^{\hspace{1pt}n-1}$ has a closed range and is thus strict. We get
\begin{equation*}
\adjustbox{scale=1,center}{\begin{tikzcd}[column sep=0.7em]
 & &&&& (d^{\hspace{1pt}n})^{-1}(0)^{\flat}\arrow[r,equal]  & (\ran d^{\hspace{1pt}n-1})^{\flat}\arrow{rrrrr}{(\overline{d^{\hspace{0.5pt}n-1}})^{-1}}\arrow{rrrrr}[swap]{\sim} &&&&& \zfrac{X^{n-1}}{(d^{\hspace{0.5pt}n-1})^{-1}(0)}\arrow[r,equal]&\zfrac{X^{n-1}}{\ran d^{\hspace{0.5pt}n-2}}\arrow[llllllldd, dashed, bend left=15, "h"] \\[-5pt]
  X^{\hspace{0.5pt}n-2}\arrow{rrrr}{d^{\hspace{0.5pt}n-2}}\arrow[bend left =25]{rrrrru}{0}\arrow[bend right =25]{rrrrrd}[swap]{0} & && & X^{\hspace{1pt}n-1}\arrow{ur}{p^{\hspace{0.5pt}n-1}}\arrow{dr}[swap]{c} & &\\[-5pt]
  &  &&&&  C &
\end{tikzcd}}
\end{equation*}
as $(\overline{d^{\hspace{1pt}n-1}})^{-1}\circ p^{\hspace{0.5pt}n-1}$ equals the quotient map $X^{\hspace{0.5pt}n-1}\rightarrow X^{\hspace{0.5pt}n-1}/(d^{\hspace{1pt}n-1})^{-1}(0)$ and is thus a cokernel of $d^{\hspace{1pt}n-2}$. This finishes part (i). Notice that for this direction of the equivalence, we used the equalities in degree $n$ and $n-1$ to conclude acyclicity in degree $n$.

\smallskip

\textcircled{3} It remains to show \textquotedblleft{}$\Longleftarrow$\textquotedblright{} for (ii) and (iii). For this purpose let $\EE\in\{\Emax,\Etop\}$ and let $n\in\ZZ$ be fixed. Notice first that $(i^{\hspace{0.5pt}n-1},\cok i^{\hspace{0.5pt}n-1})\in\Dmax$ holds with $\ran i^{\hspace{0.5pt}n-1}=(d^{\hspace{1pt}n})^{-1}(0)$ algebraically. Having the corresponding topological equality in degree $n$ thus yields $(i^{\hspace{0.5pt}n-1},\cok i^{\hspace{0.5pt}n-1})\in\EE$ meaning that $i^{\hspace{0.5pt}n-1}$ is an $\EE$-inflation. To see that $p^{\hspace{0.5pt}n-1}$ is a deflation observe that $i^{\hspace{0.5pt}n-2}\colon\ker d^{\hspace{1pt}n-1}\rightarrow X^{n-1}$ is a kernel of $p^{\hspace{0.5pt}n-1}$. Using the (algebraic and topological) equality in degree $n-1$ we get that $(i^{\hspace{0.5pt}n-2},p^{\hspace{0.5pt}n-1})\in\EE$ holds.
\end{proof}

Observe that in general $\im d^{\hspace{1pt}n-1}=\overline{d^{\hspace{1pt}n-1}(X^n)}^{\flat}$ holds. The classical de\-fi\-nition for exactness, i.e., image\hspace{1pt}=\hspace{1pt}kernel in every degree is thus weaker than $\Dmax$-acyclicity. The other natural ad hoc definition, namely range\hspace{1pt}=\hspace{1pt}kernel in every degree is precisely $\Dmax$-acyclicity. The ranges of the differentials are then necessarily closed and $\Emax$- resp.~$\Etop$-acyclicity is an additional relative openness condition. As mentioned in the proof, in Proposition \ref{LB-ACYC} the implications \textquotedblleft{}$\Longrightarrow$\textquotedblright{} holds degree-wise, while the converse directions do not.

\begin{prop}\label{PROP-AC} We have\vspace{1pt}
\begin{compactitem}

\item[(i)] $\Ac(\LBr,\Cregsur)=\Ac(\LB,\Dmax)\cap\C(\LBr)$,

\vspace{3pt}

\item[(ii)] $\Ac(\LBr,\Emaxreg)=\Ac(\LB,\Emax)\cap\C(\LBr)$,

\vspace{3pt}

\item[(iii)]$\Ac(\LBr,\Etopreg)=\Ac(\LB,\Etop)\cap\C(\LBr)$,

\vspace{3pt}

\item[(iv)] $\Ac(\LBc,\Etopcom)=\Ac(\LB,\Etop)\cap\C(\LBc)$.
\end{compactitem}
\end{prop}
\begin{proof} The inclusion \textquotedblleft{}$\subseteq$\textquotedblright{} holds in all four cases in view of the exactness of the inclusion functors, see Section \ref{SEC-EX-FKT}. It remains to show \textquotedblleft{}$\supseteq$\textquotedblright{}.

\smallskip

(i) Let $X^{\bullet}\in\C(\LBr)$ be an $\Dmax$-acyclic complex, i.e.,
\begin{equation*}
\adjustbox{scale=1,center}{\begin{tikzcd}[column sep=1.3em, row sep=1.3em]
&&Z^{\hspace{0.5pt}n-1}\arrow{rd}{i^{\hspace{0.5pt}n-2}}& & & & Z^{\hspace{0.5pt}n+1}\arrow{rd}{i^{\hspace{0.5pt}n}} &  \\
\cdots\arrow{r}&X^{\hspace{0.5pt}n-2}\arrow{ru}{p^{\hspace{0.5pt}n-2}}\arrow{rr}[swap]{d^{\hspace{1pt}n-2}}&&X^{\hspace{0.5pt}n-1}\arrow{rr}[swap]{d^{\hspace{0.5pt}n-1}}\arrow{rd}[swap]{p^{\hspace{0.5pt}n-1}}& & X^{\hspace{0.5pt}n}\arrow{ru}{p^{\hspace{0.5pt}n}}\arrow{rr}{d^{\hspace{0.5pt}n}} & & X^{\hspace{0.5pt}n+1} \arrow{r}\arrow[swap]{rd}{p^{n+1}} &\cdots \\
\phantom{X}\arrow{ru}[swap]{i^{\hspace{0.5pt}n-3}}&&\phantom{X}&&Z^{\hspace{0.5pt}n}\arrow{ru}[swap]{i^{\hspace{0.5pt}n-1}} & & & & \phantom{X}
\end{tikzcd}}
\end{equation*}
with $(i^{\hspace{0.5pt}n-1},p^{\hspace{0.5pt}n})\in\Dmax$ for all $n\in\ZZ$ and all the $X^n$ are regular LB-spaces. By Proposition \ref{LBrKERCOK} the kernel objects $Z^{\hspace{0.5pt}n}=\ker d^{\hspace{1pt}n}$ are then regular, too, and as $(\LBr,\Cregsur)\hookrightarrow(\LB,\Dmax)$ reflects exactness, we get $(i^{\hspace{0.5pt}n-1},p^{\hspace{0.5pt}n})\in\Cregsur$ for all $n\in\ZZ$.

\smallskip

(ii) \& (iii) We may argue as in (i) using that both of the inclusions $(\LBr,\Emaxreg)\hookrightarrow(\LB,\Emax)$ and $(\LBr,\Etopreg)\hookrightarrow(\LB,\Etop)$ also reflect exactness.
\smallskip

(iv) If $X^{\bullet}\in\LBc$ and $(i^{\hspace{0.5pt}n-1},p^{\hspace{0.5pt}n})\in\Etop$ holds for all $n\in\ZZ$, then the kernel objects are endowed with the subspace topology and thus complete LB-spaces. Now we may proceed as above using that the inclusion $(\LBr,\Etopcom)\hookrightarrow(\LB,\Etop)$ reflects exactness.
\end{proof}

\vspace{1pt}

\section{Derived Categories}\label{SEC-7}

The derived category of an exact category was initially  defined by Neeman \cite{Nee90}, while the one-sided case was treated for the first time by Bazzoni, Crivei \cite[Section 7]{BC13}. Their investigation has recently been extended by Henrard et al.\ \cite{ HKR, HKRW, HR19, HR19a, HR20, HR24}. As notation varies over the latter sources, and often a more general setting as needed for our purposes is considered, we will first give a brief introduction to the derived category under the following from now on standing assumption.

\begin{ass}\label{ASS} For the rest of the article we will always assume that $(\mathcal{A},\CC)$ is a strongly deflation-exact and karoubian category.\diam{}
\end{ass}

Let $\K(\mathcal{A})$ be the homotopy category furnished with the triangulated structure induced by the mapping cone triangles. For $*\in\{+,-,\text{b}\}$ let $\K^{\boldsymbol{\ast}}(\mathcal{A})$ be the full triangulated subcategory of $\K(\mathcal{A})$ formed by complexes that are homotopic to left bounded, right bounded and bounded complexes, respectively. Note that, in general, $\Ac(\mathcal{A},\CC)\subseteq\K(\mathcal{A})$ might not be closed under isomorphisms, but that under our standing Assumption \ref{ASS} it is even a thick triangulated subcategory, see \cite[Prop 3.11]{HR19}, and thus it is in particular replete. For $*\in\{+,-,\text{b}\}$ we put $\Ac^{\boldsymbol{\ast}}(\mathcal{A},\CC):=\Ac(\mathcal{A},\CC)\cap\K^{\boldsymbol{\ast}}(\mathcal{A})$ which is a thick subcategory of $\K^{\boldsymbol{\ast}}(\mathcal{A})$ as well.

\smallskip

Now we can define the derived category; for details on localizations and triangulated categories we refer to \cite{GZ, Krause, Mili,  Neeman, Verdier}.

\begin{dfn}\label{DFN-DC} Let $(\mathcal{A},\CC)$ be a strongly deflation-exact karoubian category and let $*\in\{+,-,\text{b},\emptyset\}$. We define the derived category of $(\mathcal{A},\CC)$ by the Verdier quotient
$$
\DD^{\boldsymbol{\ast}}(\mathcal{A},\CC)=\K^{\boldsymbol{\ast}}(\mathcal{A})/\Ac^{\boldsymbol{\ast}}(\mathcal{A},\CC)=\K^{\boldsymbol{\ast}}(\mathcal{A})[\mathcal{N}^{\hspace{1pt}*}(\mathcal{A},\CC)^{-1}],
$$
where $\mathcal{N}^{\hspace{1pt}*}(\mathcal{A},\CC):=\bigl\{f^{\bullet}\colon X^{\bullet}\rightarrow Y^{\bullet}\:\big|\: \cone(f^{\bullet})\in\Ac^{\boldsymbol{\ast}}(\mathcal{A},\CC)\bigr\}$ is the system of quasi-isomorphisms. We furnish the derived category with the natural triangulated structure.\diam{}
\end{dfn}

Below we summarize what is standard in the abelian case and remains valid under the assumptions on $(\mathcal{A},\CC)$ that we made above.

\begin{rmk}\label{RMK-DC}\cite[Prop 3.20, Prop 6.2, Rmk 3.5(2) and Prop 6.3]{HR19} Let $(\mathcal{A},\CC)$ be a strongly deflation-exact karoubian category and $*\in\{+,-,\text{b},\emptyset\}$.\vspace{3pt}
\begin{compactitem}

\item[(i)] The natural functor $i\colon\mathcal{A}\rightarrow\DD^{\boldsymbol{\ast}}(\mathcal{A},\CC)$, that maps $X\in\mathcal{A}$ to the stalk complex with $X$ in degree zero, is fully faithful.

\vspace{3pt}

\item[(ii)] A sequence $X\rightarrow Y\rightarrow Z$ in $\mathcal{A}$ is a conflation iff $i(X)\rightarrow i(Y)\rightarrow i(Z)\rightarrow i(X)[1]$ is a triangle in $\DD^{\boldsymbol{\ast}}(\mathcal{A},\CC)$.

\vspace{3pt}

\item[(iii)] A sequence $X\rightarrow Y\rightarrow Z$ in $\mathcal{A}$ is a conflation iff the complex $\cdots\rightarrow 0\rightarrow X\rightarrow Y\rightarrow Z\rightarrow 0\rightarrow\cdots$ is $\CC$-acyclic.

\vspace{3pt}

\item[(iv)] If $(\mathcal{A},\CC)$ is exact, then $i(\mathcal{A})\subseteq\DD^{\boldsymbol{\ast}}(\mathcal{A},\CC)$ is extension closed.\diam{}
\end{compactitem}
\end{rmk}


From our previous sections we  obtain, for $*\in\{+,-,\text{b},\emptyset\}$, the following natural functors on the derived level:

\vspace{5pt}

\adjustbox{scale=0.925,center}{\begin{tikzcd}[column sep =1.5em]
\DD^{\boldsymbol{\ast}}(\LB,\Etop)\arrow{r}{\not\sim}\arrow{r}[swap]{I_1^*}& \DD^{\boldsymbol{\ast}}(\LB,\Emax)\arrow{r}{\not\sim}\arrow{r}[swap]{I_2^*}  & \DD^{\boldsymbol{\ast}}(\LB,\Dmax)\\[9pt]
\DD^{\boldsymbol{\ast}}(\LBr,\Etopreg)\arrow{u}{G^*_{\mathsf{top}}}\arrow{r}{\not\sim}\arrow{r}[swap]{I_3^*} & \DD^{\boldsymbol{\ast}}(\LBr,\Emaxreg)\arrow{r}{\not\sim}\arrow{r}[swap]{I_4^*}\arrow[swap]{u}{G^*_{\mathsf{max}}} &\DD^{\boldsymbol{\ast}}(\LBr,\Cregsur)\arrow[swap]{u}{G^*_{\mathsf{def}}}\\[9pt]
\DD^{\boldsymbol{\ast}}(\LBc,\Etopcom)
\arrow{u}{F^*_{\mathsf{top}}}\arrow{r}{\not\sim}\arrow[swap]{r}{I_5^*}& \DD^{\boldsymbol{\ast}}(\LBc,\Emaxcom)
\arrow[swap]{u}{F^*_{\mathsf{max}}}
\arrow[swap]{r}{I_6^*}  & \DD^{\boldsymbol{\ast}}(\LBc,\Dcom).\arrow[swap]{u}{F^*_{\mathsf{def}}}
\end{tikzcd}}
\begin{picture}(0,0)(0,0)
\put(0,61.3){\begin{minipage}{447.5pt}
\begin{equation}\label{EXACT-FUN-DER}
\end{equation}
\end{minipage}}
\end{picture}

\vspace{1pt}

The non-equivalences follow from the next result, whose fairly standard proof we include for the sake of completeness, , cf.\ also Remark \ref{RMK-NEQUIV}. The vertical functors between the bottom rows being well-defined yield the desired homologification of Question \ref{Q-G} that was stated in the introduction.

\begin{prop}\label{VRT-UNF} Let $*\in\{+,-,\text{b},\emptyset\}$. The functors $I_k^*$ for $k=1,\dots,5$ in \eqref{EXACT-FUN-DER} are neither full nor faithful.
\end{prop}
\begin{proof} In all five cases we have already seen that the conflation structure on the left is a proper subclass of the conflation structure on the right. We establish that for any karoubian category $\mathcal{A}$ with a strongly deflation-exact structure $\CC$ and an exact structure $\BB\subset\CC$, the induced functor $\DD^{\boldsymbol{\ast}}(\mathcal{A},\BB)\rightarrow \DD^{\boldsymbol{\ast}}(\mathcal{A},\CC)$ can neither be full nor faithful. We begin by picking a conflation $X\stackrel{\scriptscriptstyle f}{\rightarrow}Y\stackrel{\scriptscriptstyle g}{\rightarrow}Z$ in $\CC\backslash\BB$.

\smallskip

\textcircled{1} We extend $(f,g)$ with zeros and use Remark \ref{RMK-DC}(iii) to get a complex $U^{\bullet}$ which is $\CC$-acyclic but not $\BB$-acyclic. The right roof
$$
U^{\bullet}\xrightarrow{\hspace{2pt}\id_{\hspace{0.5pt}\scalebox{0.5}{$U^{\bullet}$}}}U^{\bullet}\mathop{\xleftarrow{\,\id_{\hspace{0.5pt}\scalebox{0.5}{$U^{\bullet}$}}}}_{\raisebox{3pt}{$\scriptscriptstyle\sim$}}U^{\bullet}
$$
then represents the zero morphism in $\DD^{\boldsymbol{\ast}}(\mathcal{A},\CC)$ but in $\DD^{\boldsymbol{\ast}}(\mathcal{A},\BB)$ it does not: Let us assume it does. Then there exists $\sigma\colon U^{\bullet}\rightarrow V^{\bullet}$, $\sigma\in\mathcal{N}^{\hspace{1pt}*}(\mathcal{A},\CC)$ with $\sigma\circ\id_{U^{\bullet}}=0$, see e.g., \cite[Lem 2.1.5 on p.~37]{Mili}. It follows $\sigma=0$ and this means $U^{\bullet}\cong0$ in $\DD^{\boldsymbol{\ast}}(\mathcal{A},\BB)$. 

\smallskip

\textcircled{2} We apply Remark \ref{RMK-DC}(ii) to $(f,g)$, drop the inclusion $i$ into the derived category from the notation, and obtain a triangle of stalk complexes
\begin{equation}\label{TR-1}
X\longrightarrow Y\longrightarrow Z\stackrel{u}{\longrightarrow} X[1]
\end{equation}
in $\DD^{\boldsymbol{\ast}}(\mathcal{A},\CC)$, which is not a triangle in $\DD^{\boldsymbol{\ast}}(\mathcal{A},\BB)$. Let us assume that the functor $\DD^{\boldsymbol{\ast}}(\mathcal{A},\BB)\rightarrow \DD^{\boldsymbol{\ast}}(\mathcal{A},\CC)$ is full. Then we have
$$
\Hom_{\DD^{\boldsymbol{\ast}}(\mathcal{A},\BB)}(Z,X[1])=\Hom_{\DD^{\boldsymbol{\ast}}(\mathcal{A},\CC)}(Z,X[1]),
$$
which means that $u\colon Z\rightarrow X[1]$ is a morphism in $\DD^{\boldsymbol{\ast}}(\mathcal{A},\BB)$. Using TR1, TR2 and $[-1]$ we get a triangle $X\rightarrow C^{\bullet}\rightarrow Z\rightarrow X[1]$ in $\DD^{\boldsymbol{\ast}}(\mathcal{A},\BB)$. As $(\mathcal{A},\BB)$ is exact and karoubian, there exists $B\in\mathcal{A}$ with $C^{\bullet}\cong B$ in $\DD^{\boldsymbol{\ast}}(\mathcal{A},\BB)$ by Remark \ref{RMK-DC}(iv). Using TR3 and TR1 we obtain that the triangles \eqref{TR-1} and
$$
X\longrightarrow B\longrightarrow Z\stackrel{u}{\longrightarrow} X[1]
$$
are isomorphic in $\DD^{\boldsymbol{\ast}}(\mathcal{A},\BB)$ and thus in particular $Y\cong B$ must hold. As both are stalk complexes, $Y$ and $B$ are isomorphic as objects in $\mathcal{A}$. By Remark \ref{RMK-DC}(ii) the conflation $X\rightarrow Y\rightarrow Z$ thus belongs to $\BB$.
\end{proof}

\begin{rmk}\label{RMK-NEQUIV}\begin{myitemize}\setlength{\itemindent}{-10pt}\item[(i)] That $I_1^*$ and $I_2^*$ are unfaithful was proven already in \cite[Thm 9.6]{HKRW}; the idea for the non-fullness proof above is due to van Roosmalen \cite{Adam24}. That none of the five functors is a triangle equivalence can be concluded alternatively with Remark \ref{RMK-DC}(ii).

\vspace{3pt}\setlength{\itemindent}{0pt}

\item[(ii)] As $I_k^*$ for $k=1,2,3$ is a functor between exact categories, one may show non-fullness in this case without the use of triangles as follows: If $\DD^{\boldsymbol{\ast}}(\mathcal{A},\BB)\rightarrow\DD^{\boldsymbol{\ast}}(\mathcal{A},\CC)$ were full, then by \cite[Prop in A.7]{Posi} the Yoneda Ext-groups
$$
\Ext^1_{(\mathcal{A},\CC)}(Z,X)=\Ext^1_{(\mathcal{A},\BB)}(Z,X)
$$
would coincide and there could thus not exist $X\stackrel{\scriptscriptstyle f}{\rightarrow}Y\stackrel{\scriptscriptstyle g}{\rightarrow}Z$ in $\CC\backslash\BB$.

\vspace{3pt}

\item[(iii)] The properties of $I_6^*$ are unknown.\diam{}

\end{myitemize}
\end{rmk}

In order to study the vertical functors in \eqref{EXACT-FUN-DER} we need to extend the notion of the `derived category of $\mathcal{B}\subseteq\mathcal{A}$ relative to $\mathcal{A}$' as used in \cite[Dfn 3.6]{SKT11} for $\mathcal{A}$ abelian and $\mathcal{B}\subseteq\mathcal{A}$ a full additive subcategory to our situation, namely $(\mathcal{A},\CC)$ being strongly deflation-exact and karoubian. To do so we observe that, using Remark \ref{RMK-DC}, it follows that $\Ac^{\boldsymbol{\ast}}(\mathcal{B};\mathcal{A},\CC):=\Ac(\mathcal{A},\CC)\cap\K^{\boldsymbol{\ast}}(\mathcal{B})$ is a thick triangulated subcategory of $\K^{\boldsymbol{\ast}}(\mathcal{B})$ for $*\in\{+,-,\text{b},\emptyset\}$.

\begin{dfn}\label{Stovi-DFN} Let $(\mathcal{A},\CC)$ be a strongly deflation-exact karoubian category, let $\mathcal{B}\subseteq\mathcal{A}$ be a full additive subcategory and let $*\in\{+,-,\text{b},\emptyset\}$. The Verdier quotient
$$
\DD^{\boldsymbol{\ast}}(\mathcal{B};\mathcal{A},\CC):=\K^{\boldsymbol{\ast}}(\mathcal{B})/\Ac^{\boldsymbol{\ast}}(\mathcal{B};\mathcal{A},\CC)=\K^{\boldsymbol{\ast}}(\mathcal{B})[\mathcal{N}^{\hspace{1pt}*}(\mathcal{B};\mathcal{A},\CC)^{-1}],
$$
with $\mathcal{N}^{\hspace{1pt}*}(\mathcal{B};\mathcal{A},\CC):=\bigl\{f^{\bullet}\colon X^{\bullet}\rightarrow Y^{\bullet}\:\big|\: \cone(f^{\bullet})\in\Ac^{\boldsymbol{\ast}}(\mathcal{B};\mathcal{A},\CC)\bigr\}$, defines the \emph{derived category of $\mathcal{B}$ relative to $(\mathcal{A},\CC)$}. \diam{}
\end{dfn}

To avoid notational overkill we write $\DD^{\boldsymbol{\ast}}(\mathcal{B},\CC)$ instead of $\DD^{\boldsymbol{\ast}}(\mathcal{B};\mathcal{A},\CC)$ provided that it is clear from context that $\CC$ is a conflation structure on $\mathcal{A}$. Indeed, this is the case in all our concrete examples as our notation, e.g., $\CC=\Etop$, always indicates on which category the conflation structure is defined. Notice however that even if $\mathcal{B}\subseteq(\mathcal{A},\CC)$ is a fully exact subcategory, $\DD^{\boldsymbol{\ast}}(\mathcal{B},\CC\cap\mathcal{B})$ and $\DD^{\boldsymbol{\ast}}(\mathcal{B},\CC)$ might \emph{not} be the same.

\begin{rmk}\label{EQ-REM} If $(\mathcal{A},\CC)$ is a strongly deflation-exact karoubian category and $\mathcal{B}\subseteq\mathcal{A}$ a full additive subcategory endowed with a conflation structure $\BB\subseteq\CC\cap\mathcal{B}$ such that $(\mathcal{B},\BB)$ is strongly deflation-exact and karoubian, then the natural functor between the `usual' derived categories factors through the `relative' derived category:
$$
\DD^{\boldsymbol{\ast}}(\mathcal{B},\BB)\longrightarrow\DD^{\boldsymbol{\ast}}(\mathcal{B};\mathcal{A},\CC)\longrightarrow \DD^{\boldsymbol{\ast}}(\mathcal{A},\CC).
$$
In particular, if $\Ac(\mathcal{A},\CC)\cap\K^{\boldsymbol{\ast}}(\mathcal{B})=\Ac^{\boldsymbol{\ast}}(\mathcal{B},\BB)$ holds, then the first functor above is an equivalence. By Proposition \ref{PROP-AC} this applies to $G^{*}_{\mathsf{def}}$, $G^{*}_{\mathsf{max}}$, $G^{*}_{\mathsf{top}}$ and $F^{*}_{\mathsf{top}}$.\diam{}
\end{rmk}

\begin{thm}\label{MAIN-1} For $*\in\{-,\textrm{b}\}$ the following are triangle equivalences:\vspace{2pt}
\begin{myitemize}

\item[(i)] $\DD^{\boldsymbol{\ast}}(\LBr,\Cregsur)\rightarrow\DD^{\boldsymbol{\ast}}(\LB,\Dmax)$,

\vspace{3pt}

\item[(ii)] $\DD^{\boldsymbol{\ast}}(\LBc,\Cregsur)\rightarrow\DD^{\boldsymbol{\ast}}(\LBr,\Cregsur)$.
\end{myitemize}
\end{thm}
\begin{proof} By \eqref{RES}, we get for every LB-space $X=\ind_{n\in\NN}X_n$ the complex
\begin{equation*}
\cdots\longrightarrow 0\longrightarrow\Bigosum{n\in\mathbb{N}}{}X_n\stackrel{d}{\longrightarrow}\Bigosum{n\in\mathbb{N}}{}X_n\stackrel{\sigma}{\longrightarrow}X\longrightarrow0\longrightarrow \cdots
\end{equation*}
where the direct sums are complete LB-spaces and which is $\Dmax$-acyclic by Remark \ref{RMK-DC}(iv). If $X$ is regular, then the complex is $\Cregsur$-acyclic. Moreover, by Remark \ref{EQ-REM}, the domain in (i) equals
$$
\DD^{\boldsymbol{\ast}}(\LBr,\Cregsur)=\DD^{\boldsymbol{\ast}}(\LBr,\Dmax)=\DD^{\boldsymbol{\ast}}(\LBr;\LB,\Dmax).
$$
In (ii) it holds $\DD^{\boldsymbol{\ast}}(\LBc,\Cregsur)=\DD^{\boldsymbol{\ast}}(\LBc;\LBr,\Cregsur)$ by definition. We thus may prove (i) and (ii) by considering the following abstract situation: Let $(\mathcal{A},\CC)$ be a strongly deflation exact karoubian category and let $\mathcal{B}\subseteq\mathcal{A}$ be a full additive subcategory. Assume that for every $A\in\mathcal{A}$ there exists a $\CC$-acyclic complex
\begin{equation}\label{B-RES}
\cdots\longrightarrow 0\longrightarrow B^{-n}\longrightarrow B^{-n+1}\longrightarrow\cdots\longrightarrow B^{0}\longrightarrow  A\longrightarrow0\longrightarrow \cdots
\end{equation}
with $B^{\hspace{1pt}j}\in\mathcal{B}$ and $n\in\NN_0$. We claim that the natural functor $\DD^{\boldsymbol{\ast}}(\mathcal{B};\mathcal{A},\CC)\rightarrow \DD^{\boldsymbol{\ast}}(\mathcal{A},\CC)$ is a triangle equivalence for $*\in\{-,\text{b}\}$.

\smallskip

\textcircled{1} Let $A^{\bullet}\in\K^{\boldsymbol{\ast}}(\mathcal{A})$ be given. Without loss of generality we assume that $A^j=0$ for $j\geqslant1$. By our assumption we find a deflation $\tau^{\hspace{1pt}0}\colon B^{\hspace{1pt}0}\rightarrow A^0$ with $B^{\hspace{1pt}0}\in\mathcal{B}$. By \hypref{R2}, the pullback $(P^{-1},p^{-1},g^{-1})$ of $\tau^{\hspace{1pt}0}$ along $d_A^{-1}$ exists and $g^{-1}$ is a deflation. By the pullback property we get a map $q^{-2}\colon A^{-2}\rightarrow P^{-1}$ with $p^{-1}\circ q^{-2}=0$ and $g^{-1}\circ q^{-2}=d_A^{-2}$. Iteration of this procedure yields the following diagram.\vspace{-5pt}
\begin{equation*}
\adjustbox{scale=1,center}{\begin{tikzcd}
& \vdots &\vdots & \vdots & \vdots& \vdots& \vdots\\[-20pt]
%
%
%
\cdots\arrow{r}&A^{-4}\arrow[d,equal]\arrow{r}&A^{-3}\arrow[d,equal]\arrow[swap, dashed]{r}{q^{-3}}&P^{-2}\arrow[swap]{d}{g^{-2}}\arrow{r}{p^{-2}}\commutes[\mathrm{PB}]{dr}&B^{-1}\arrow{r}\arrow[]{d}{h^{-1}} & B^0 \arrow[d, equal]\arrow{r} & 0\arrow{r}\arrow{d} &\cdots\\[4pt]
\cdots\arrow{r}&A^{-4}\arrow[d,equal]\arrow{r}&A^{-3}\arrow[d,equal]\arrow{r}&A^{-2}\arrow[d,equal]\arrow[swap, dashed]{r}{q^{-2}}&P^{-1}\arrow{r}{p^{-1}}\arrow[swap]{d}{g^{-1}}\commutes[\mathrm{PB}]{dr} & B^0 \arrow[]{d}{\tau^0}\arrow{r} & 0\arrow{r}\arrow{d} &\cdots\\[4pt]
\cdots\arrow{r}&A^{-4}\arrow[swap]{r}{d_A^{-4}}&A^{-3}\arrow[swap]{r}{d_A^{-3}}&A^{-2}\arrow[swap]{r}{d_A^{-2}} &A^{-1} \arrow{r}[swap]{d_A^{-1}} & A^0\arrow{r} & 0\arrow{r} &\cdots
\end{tikzcd}}
\end{equation*}
If $*=-$, then we continue the above indefinitely and obtain a complex $B^{\bullet}\in\C^{\boldsymbol{-}}(\mathcal{B})$ with objects $B^{\hspace{1pt}j}\in\mathcal{B}$ and as differentials we put $d_B^{\hspace{1pt}j}:=-p^{\hspace{1pt}j}\circ h^{\hspace{1pt}j}$ for $j\leqslant 0$ and zero for $j\geqslant1$. If $*=\text{b}$, then there exists $k\in\NN$ such that $A^{j}=0$ holds for all $j<-k$. Thus, after the $k$-th pullback, we may take a resolution of $P^{-k}$ as in \eqref{B-RES} with only finitely many non-zero terms and obtain
\begin{equation*}
\adjustbox{scale=1,center}{\begin{tikzcd}
\cdots\arrow{r}&B^{-k-2}\arrow{d}\arrow{r}{d_B^{-k-2}}&B^{-k-1}\arrow{d}\arrow{r}{d_B^{-k-1}}&B^{-k}\arrow[swap]{d}{h^{-k}}\arrow{r}{}&B^{-k+1}\arrow{r}{}\arrow[equal]{d} & B^{-n+2} \arrow[equal]{d}\arrow{r} &\cdots\\[4pt]
\cdots\arrow{r}&0\arrow{d}\arrow{r}&0\arrow{d}\arrow{r}&P^{-k}\commutes[\mathrm{PB}]{dr}\arrow[swap]{d}{g^{-k}}\arrow{r}{p^{-k}}&B^{-k+1}\arrow{r}\arrow{d}{h^{-k+1}} & B^{-k+2} \arrow[equal]{d}\arrow{r} &\cdots\\[4pt]
\cdots\arrow{r}&0\arrow[swap]{r}{}&0\arrow[swap]{r}{}&A^{-k}\arrow[swap, dashed]{r}{q^{-k}} &P^{-k+1} \arrow{r}[]{p^{-k+1}} & B^{-k+2}\arrow{r} \arrow{r} &\cdots
\end{tikzcd}}
\end{equation*}
where the top row defines a complex $B^{\bullet}\in\C^{\bb}(\mathcal{B})$ with differentials for $j<-k$ as given and for $j\geqslant -k$ as defined in the $*=-$ case.   We define now $\tau\colon B^{\bullet}\rightarrow A^{\bullet}$; in the case $*=-$ via $\tau^{\hspace{1pt}j}:=(-1)^{\hspace{0.5pt}j}\cdot g^{\hspace{1pt}j}\circ h^{\hspace{1pt}j}$ for $j<0$ and zero for $j>0$, in the case $*=\text{b}$ we put $\tau^{\hspace{1pt}j}=0$ for $j<-k$. We calculate the mapping cone of $\tau$ and get
\begin{equation*}
\adjustbox{scale=1,center}{\begin{tikzcd}[column sep=0.6cm, row sep =0.4cm, ampersand replacement=\&]
\& \& P^{-2}\arrow{rd}{\scalebox{0.75}{$\bigl[\begin{smallmatrix}p^{-2}\\g^{-2}\end{smallmatrix}\bigr]$}}\\
\cdots\arrow{r}\&[-0.8em]B^{-2}\oplus{}A^{-3}\arrow{ru}{\scalebox{0.70}{$\bigl[\hspace{0.5pt}h^{-2}\;q^{-3}\bigr]$}}\arrow{rr}[swap]{\scalebox{0.75}{$
\Bigl[\begin{smallmatrix}
-d_B^{-2} & 0\\
\tau^{-2} & d_{A}^{-3}
\end{smallmatrix}\Bigr]
$}}\&\&B^{-1}\oplus{}A^{-2}\arrow[swap]{rd}{\scalebox{0.75}{$\bigl[h^{-1}\;-q^{-2}\bigr]$}}\arrow{rr}{\scalebox{0.75}{$
\Bigl[\begin{smallmatrix}
-d_B^{-1} & 0\\
\tau^{-1} & d_{A}^{-2}
\end{smallmatrix}\Bigr]
$}}\&\&B^{0}\oplus{}A^{-1}\arrow{rr}{\hspace{-5pt}[\hspace{1pt}\scalebox{0.75}{$
\tau^{-1}\;d_{A}^{-2}
$}\hspace{1pt}]}\&\&A^{0}\arrow{r}\&0\\[3pt]
\phantom{X^i}\arrow{ru}\&\& \&\&P^{-1}\arrow[swap]{ru}{\scalebox{0.75}{$\bigl[\begin{smallmatrix}\phantom{-}p^{-1}\\-g^{-1}\end{smallmatrix}\bigr]$}} \&\&
\end{tikzcd}}
\end{equation*}
where the acyclicity follows from \cite[dual of Prop 5.7]{BC13}. In the case of $*=-$ we thus get that $\tau\colon B^{\bullet}\rightarrow A^{\bullet}$ belongs to $\mathcal{N}^{\hspace{1pt}-}(\mathcal{A},\CC)$. If $*=\text{b}$ then the left end of $\cone(\tau)$ looks like
\begin{equation*}
\adjustbox{scale=1,center}{\begin{tikzcd}[column sep=0.4cm, row sep =0.4cm, ampersand replacement=\&]
\& \& \hspace{10pt}\ker h^{-k}\hspace{-3pt}\arrow{rd}{i^{-k}}\\
\cdots\arrow{r}\&[0.5em]B^{-k-1}\arrow{rr}[swap]{d_B^{-k-1}}\arrow{ru}{p^{-k-1}}\&\&[0.4em]B^{-k}\arrow{rr}{\scalebox{0.75}{$
\Bigl[\begin{smallmatrix}
-d_B^{-k}\\
\tau^{-k} 
\end{smallmatrix}\Bigr]
$}}\arrow[swap]{rd}{h^{-k}}\&\&B^{-k+1}\oplus{}A^{-k}\arrow{rd}\arrow{r}\&\cdots\\[4pt]
\phantom{K^I}\arrow[swap]{ru}{}\&\&  \&\& P^{-k}\arrow[swap]{ru}{\scalebox{0.75}{$\bigl[\begin{smallmatrix}\phantom{(-1)^{k}}p^{-k}\\(-1)^{k}g^{-k}\end{smallmatrix}\bigr]$}}\&\& \phantom{\cdots}
\end{tikzcd}}
\end{equation*}
where we get inflations/deflations $h^k$, $i^{-k}$, $p^{-k-1}$ etc.\ from \eqref{B-RES}. The remaining maps are deflations/inflations due to \cite[dual of Prop 5.7]{BC13}. Our construction thus yields $\tau\colon B^{\bullet}\rightarrow A^{\bullet}$ with $B^{\bullet}\in\C^{\bb}(\mathcal{B})$ and $\tau\in\mathcal{N}^{\hspace{1pt}\text{b}}(\mathcal{A},\CC)$. In particular, $\DD^{\boldsymbol{\ast}}(\mathcal{B};\mathcal{A},\CC)\rightarrow\DD^{\boldsymbol{\ast}}(\mathcal{A},\CC)$ is essentially surjective.

\smallskip

\textcircled{2} Let $X^{\bullet}\stackrel{\scriptscriptstyle f}{\rightarrow}Z^{\bullet}\stackrel{\scriptscriptstyle\sigma}{\leftarrow}Y^{\bullet}$ be in $\Hom_{\DD^{\boldsymbol{\ast}}(\mathcal{B};\mathcal{A},\CC)}(X^{\bullet},Y^{\bullet})$ and assume that this roof represents the zero morphism in $\DD^{\boldsymbol{\ast}}(\mathcal{A},\CC)$. Then there exists $s\colon A^{\bullet}\rightarrow X^{\bullet}$ in $\K^{\boldsymbol{\ast}}(\mathcal{A})$ with $f\circ s=0$, see, e.g., \cite[Lem 1.2.5 on p.~37]{Mili}. By \textcircled{1} there exists $\tau\colon B^{\bullet}\rightarrow A^{\bullet}$ with $B^{\bullet}\in\K^{\boldsymbol{\ast}}(\mathcal{B})$ and $\tau\in\mathcal{N}^{\hspace{1pt}*}(\mathcal{A},\CC)$. Then $s\circ\tau\in\mathcal{N}^{\hspace{1pt}*}(\mathcal{A},\CC)$ and as $B^{\bullet}$ and $Z^{\bullet}$ are complexes over $\mathcal{B}$ we get $s\circ\tau\in\mathcal{N}^{\hspace{1pt}*}(\mathcal{B};\mathcal{A},\CC)$, $f\circ s\circ\tau=0$ and thus our initial roof was already zero in $\Hom_{\DD^{\boldsymbol{\ast}}(\mathcal{B};\mathcal{A},\CC)}(X^{\bullet},Y^{\bullet})$.

\smallskip

\textcircled{3} Let now $X^{\bullet},Y^{\bullet}\in\K^{\boldsymbol{\ast}}(\mathcal{B})$ and let $X\stackrel{\scriptscriptstyle f}{\rightarrow}Z^{\bullet}\stackrel{\scriptscriptstyle\sigma}{\leftarrow}Y\in\Hom_{\DD^{\boldsymbol{\ast}}(\mathcal{A},\CC)}(X^{\bullet},Y^{\bullet})$. By \textcircled{1} there exists $\tau\colon Z^{\bullet}\rightarrow B^{\bullet}$ with $B^{\bullet}\in\K^{\boldsymbol{\ast}}(\mathcal{B})$ and $\tau\in\mathcal{N}^{\hspace{1pt}*}(\mathcal{A},\CC)$. As above it follows $\sigma\circ\tau\in\mathcal{N}^{\hspace{1pt}*}(\mathcal{B};\mathcal{A},\CC)$ and the initial roof is equivalent to $X^{\bullet}\xrightarrow{\scriptscriptstyle \tau\circ f}B^{\bullet}\xleftarrow{\scriptscriptstyle \tau\circ\sigma}Y^{\bullet}$ which belongs to $\Hom_{\DD^{\boldsymbol{\ast}}(\mathcal{B};\mathcal{A},\CC)}(X^{\bullet},Y^{\bullet})$. 
\end{proof}

\begin{rmk}\begin{myitemize}\setlength{\itemindent}{-10pt}\item[(i)] The proof of \textcircled{1} goes back to \cite[Prop I.4.6]{H66} and has been generalized and modified by many authors over the years. Our proof follows \cite[Lem 3.12]{HR20} but  establishes the quasi-isomorphism property of $\tau$ in a slightly different way. Parts \textcircled{2} and \textcircled{3} are based on \cite[Lem 3.8]{SKT11}, \cite[I.3.3]{HRS96} and \cite[Prop 7.2.1(ii)]{KS} where however abelian categories were considered.

\smallskip\setlength{\itemindent}{0pt}

\item[(ii)] In the notation of \cite[Dfn 3.1]{HR20} the category $\LBr$ is a \emph{uniformly preresolving subcategory} of $(\LB,\Dmax)$; Theorem \ref{MAIN-1}(i) can alternatively be derived from \cite[Thm 3.11]{HR20} and by \cite[Thm 4.1]{HR20} its statement holds for $*=\emptyset$, too. If $\LBc\subseteq(\LBr,\Cregsur)$ is (uniformly) preresolving is unknown.\diam{}

\end{myitemize}
\end{rmk}

\begin{thm}\label{DE-COR} For $*\in\{-,\text{b}\}$ the following natural functor is a triangle equivalence:
$$
F^*_{\mathsf{rel}}\colon\DD^{\boldsymbol{\ast}}(\LBc,\Dmax)\rightarrow\DD^{\boldsymbol{\ast}}(\LBr,\Dmax).
$$
\end{thm}
\begin{proof}By using Proposition \ref{PROP-AC}(i) twice and Theorem \ref{MAIN-1}(ii) once we get
\begin{equation*}
\begin{aligned}
\DD^{\boldsymbol{\ast}}(\LBc,\Dmax)& = \zfrac{\K^{\boldsymbol{\ast}}(\LBc)}{\Ac(\LBr,\Cregsur)\cap\K^{\boldsymbol{\ast}}(\LBc)}= \DD^{\boldsymbol{\ast}}(\LBc,\Cregsur)\\
& \cong \DD^{\boldsymbol{\ast}}(\LBr,\Cregsur)=\zfrac{\K^{\boldsymbol{\ast}}(\LBr)}{\Ac(\LB,\Dmax)\cap\K^{\boldsymbol{\ast}}(\LBr)}   \\[2pt]
& = \DD^{\boldsymbol{\ast}}(\LBr,\Dmax)
\end{aligned}
\end{equation*}
as desired.
\end{proof}

Before we turn to global dimensions in the next section, we note the following byproduct of the proof of Theorem \ref{MAIN-1}: If $(\mathcal{A},\CC)$ is as in Assumption \ref{ASS} and $\mathcal{B}\subseteq\mathcal{A}$ is a full additive subcategory such that for every $A\in\mathcal{A}$ there exists a resolution as in \eqref{B-RES}, the number
$$
\resdim_{\mathcal{B}}(\mathcal{A},\CC):=\sup_{A\in\mathcal{A}}\resdim_{\mathcal{B}}(A)\in\NN_0\cup\{\infty\}
$$
is called the \emph{resolution dimension of $\mathcal{A}$ with respect to $\mathcal{B}$}, see \cite[Dfn 3.1]{HR20}. With this notation we get immediately the following.

\begin{prop}\label{RES-DIM-COR} We have
$$
\resdim_{\LBc}(\LB,\Dmax)=\resdim_{\LBr}(\LB,\Dmax)=1.
$$
\end{prop}
\begin{proof} The standard resolution shows that both numbers are less or equal to one. Since there exist LB-spaces which are not regular (and thus not complete either) the dimensions must both be greater or equal to one.
\end{proof}


\section{Global dimensions}\label{SEC-GD}

We define the global dimension of a strongly deflation-exact and karoubian category analogously to exact categories: Firstly, we say that an object $P\in\mathcal{A}$ is \emph{projective} if for any deflation $e\colon E\rightarrow X$  and any morphism $f\colon P\rightarrow X$ there exists a morphism $g\colon P\rightarrow E$ such that $e\circ g=f$. Then we say that $(\mathcal{A},\CC)$ has \emph{enough projectives} if for every $X\in\mathcal{A}$ there exists a deflation $P\rightarrow X$ with $P$ projective. A \emph{projective resolution} of $X\in\mathcal{A}$ is a $\CC$-acyclic complex
$$
\cdots\longrightarrow P^{\hspace{0.5pt}-n}\longrightarrow \cdots\longrightarrow P^{-1}\longrightarrow P^{\hspace{0.5pt}0}\longrightarrow X\longrightarrow 0
$$
with $P^{\hspace{1pt}j}$ projective. If $X$ has a finite projective resolution, then the smallest possible $n\in\NN$ for which a resolution with $P^{\hspace{0.5pt}-n}\not=0$ and $P^{\hspace{1pt}j}=0$ for $j<-n$ exists, is denoted by $\operatorname{pd}X$ and called the \emph{projective dimension} of $X$. If $X$ has only infinite projective resolutions, we put $\operatorname{pd}X=\infty$. Given that $(\mathcal{A},\CC)$ has enough projectives, every object has a projective resolution and the \emph{global dimension} of $(\mathcal{A},\CC)$ is defined as
$$
\gldim(\mathcal{A},\CC):=\sup_{\hspace{-1pt}X\in\mathcal{A}\hspace{2pt}}\operatorname{pd}X\in\NN_0\cup\{\infty\}.
$$
Note that in this case $\gldim(\mathcal{A},\CC)=\resdim_{\Proj(\mathcal{A},\CC)}(\mathcal{A},\CC)$ holds, where the resolution dimension is defined as at the end of Section \ref{SEC-7} and where $\Proj(\mathcal{A},\CC)$ denotes the full subcategory formed by the projective objects.

\smallskip

The result below on the projective dimension of $(\LB,\Dmax)$ follows essentially by refining arguments of K\"othe \cite{KoetheHeb66} and Doma\'nski \cite{DomProj92}. Notice that in particular the classification of projectives in Theorem \ref{PROP-PROJ}(i) is stated exactly as below already in \cite[Main Thm]{DomProj92} but that Doma\'nski defined projective LB-spaces not in the way that we did above. We will however show that his definition is indeed equivalent to $\Dmax$-projectivity.

\begin{thm}\label{PROP-PROJ} We consider the category $(\LB,\Dmax)$. \vspace{2pt}
\begin{compactitem}
\item[(i)] An object $P$ is projective iff $P\cong\bigoplus_{n\in\mathbb{N}}\ell^{\hspace{0.5pt}1}(\Gamma_n)$ for suitable $\Gamma_n$.\vspace{3pt}

\item[(ii)] The category has enough projectives.\vspace{3pt}

\item[(iii)] We have $\gldim(\LB,\Dmax)=1$.
\end{compactitem}
\end{thm}
\begin{proof}\textcircled{1} Let us first show that in any deflation-exact category $(\mathcal{A},\CC)$ an object $P$ is projective if and only if every conflation
\begin{equation}\label{CONF}
A\longrightarrow B\stackrel{q}{\longrightarrow}P
\end{equation}
splits. Indeed, if $P$ is projective, then $\id_P$ lifts along the deflation $q$ yielding $r\colon P\rightarrow B$ with $q\circ r=\id_P$. For the other direction let an arbitrary $f\colon P\rightarrow X$ and a deflation $e\colon E\rightarrow X$ be given. We form the pullback of $e$ along $f$. Due to \hypref{R2} we get a deflation and can take its kernel:
\begin{equation*}
\begin{tikzcd}
\ker i_P\arrow{r}{}&S\arrow{r}{i_P}\arrow{d}[swap]{i_T}\commutes[\mathrm{PO}]{dr} & P \arrow{d}{f}\\[4pt]
&T \arrow{r}[swap]{e} & X.
\end{tikzcd}
\end{equation*}
By assumption, the conflation $\ker i_P\longrightarrow S\stackrel{i_P}{\longrightarrow}P$ splits and  there exists $r\colon P\rightarrow S$ with $i_P\circ r=\id_P$. It follows $e\circ i_T\circ r = f\circ i_P\circ r = f$.

\smallskip

\textcircled{2} Next we note that, for a given class $\mathcal{R}$ of lcs, K\"othe \cite[p.~182]{KoetheHeb66} calls a space $P$ \emph{liftable in the class $\mathcal{R}$} (in K\"othe's German-language article: `hebbar in der Klasse $\mathcal{R}$'), if for every $B\in\mathcal{R}$ such that $B/A\cong P$ holds for some subspace $A\subseteq B$, there is a continuous projection onto $P$ with kernel $A$. An LB-space $P$ is thus liftable in the class $\LB$ iff every $\LB$-kernel-cokernel pair of the form $A\rightarrow B\rightarrow P$ splits. By \textcircled{1} and \cite[6(3)]{KoetheHeb66} any direct sum of spaces of the form $\ell^{\hspace{0.5pt}1}(\Gamma)$ is thus projective. This shows the implication \textquotedblleft{}$\Longleftarrow$\textquotedblright{} in (i).

\smallskip

\textcircled{3} Let now $X=\ind_{n\in\NN}X_n$ be an LB-space with defining sequence $X_0\hookrightarrow X_1\hookrightarrow\cdots$. Denote by $B_n$ the closed unit ball of $X_n$. By successive renorming, if necessary, we achieve that $B_n\subseteq B_{n+1}$ holds for $n\in\NN$. We define
$$
S\colon \Bigosum{n\in\NN}{}\ell^1(B_n)\rightarrow X,\;\bigl((\lambda_{b,n})_{b\in B_n}\bigr)_{n\in\NN}\mapsto \fsum_{n\in\NN}\bigl(\fsum_{b\in B_n}\lambda_{b,n}b\bigr)
$$
which is a linear, continuous and surjective map. Moreover, we define
$$
D\colon \Bigosum{n\in\NN}{}\ell^1(B_n)\rightarrow\Bigosum{n\in\NN}{}\ell^1(B_n),\,\bigl((\lambda_{b,n})_{b}\bigr)_{n}\mapsto \bigl((\lambda_{b,n})_{b\in B_n}-(\lambda_{b,n-1})_{b\in B_{n-1}}\bigr)_{n\in\NN}
$$
where we use $B_{n-1}\subseteq B_n$ to read $(\lambda_{b,n-1})_{b\in B_{n-1}}\in\ell^1(B_n)$ by putting $\lambda_{b,n-1}:=0$ for $b\in B_{n}\backslash B_{n-1}$ and by agreeing that $(\lambda_{b,-1})_{b\in B_{-1}}\equiv0\in\ell^1(B_0)$. Using the universal property of the direct sum it follows that $D$ is linear and continuous. A direct calculation shows that $D$ is injective. Using
$$
T\colon S^{-1}(0)\rightarrow \Bigosum{n\in\NN}{}\ell^1(B_n),\;\bigl((\lambda_{b,n})_{b}\bigr)_{n}\mapsto\bigl(\fsum_{k=0}^n(\lambda_{b,k})_{b\in B_k}\bigr)_{n\in\NN},
$$
where we again read $(\lambda_{b,k})_{b\in B_k}\in\ell^1(B_n)$ for $k=0,\dots,n-1$ via $B_k\subseteq B_n$ and fill undefined entries with zeros, we see that $(D,S)$ is a kernel-cokernel pair in $\LB$. In particular the above shows that every LB-space $X$ can be written as the quotient of a direct sum of $\ell^1(\Gamma)$s by a closed subspace. This was mentioned already in \cite[Proof of 6(4)]{KoetheHeb66} without all the details. In combination with \textcircled{2} we get that $\LB$ has enough projectives, i.e., we established (ii), and that the category's global dimension is at most one.

\smallskip

\textcircled{4} Let $P$ be a projective object of $\LB$. By \textcircled{3} we get a conflation
$$
\ker S\longrightarrow\Bigosum{n\in\NN}{}\ell^1(B_n)\stackrel{S}{\longrightarrow} P
$$
which splits as $P$ is projective. This means that $P\subseteq\bigoplus_{n\in\mathbb{N}}\ell^1(B_n)$ is a complemented subspace. Doma\'nski \cite{DomProj92} proved via Pelczy\'nski's decomposition procedure that then $P$ itself is isomorphic to a direct sum of $\ell^1(\Gamma)$s. This completes the proof of (i).

\smallskip

\textcircled{5} As direct sums of $\ell^1(\Gamma)$s are in particular regular LB-spaces, non-regular LB-spaces cannot be projective. Thus, the projective dimension of $(\LB,\Call)$ cannot be zero, which completes the proof of (iii).
\end{proof}

\begin{rmk}\label{PROJ-RMK}\begin{myitemize}\setlength{\itemindent}{-10pt}\item[(i)] For $\EE\in\{\Emax,\Etop\}$ we have $\EE\subseteq\Call$ and thus every object that is isomorphic to a countable direct sum of $\ell^1(\Gamma)$s is projective in $(\LB,\EE)$. It is unknown if these are all projectives as well as it is open if $(\LB,\EE)$ has enough projectives.

\vspace{3pt}

\setlength{\itemindent}{0pt}\item[(ii)] We point out that the category $\LB$ does not have enough projectives if the latter notion is defined not relative to a conflation structure, but by requiring that a morphism $f\colon P\rightarrow X$ out of a projective lifts along every epimorphism $e\colon E\rightarrow X$, see \cite[Thm 5.3(iv)]{HSW}.
\end{myitemize}
\end{rmk}

\begin{thm}\label{PROP-PROJ-2}For $(\LBr,\Cregsur)$ all statements from Theorem \ref{PROP-PROJ} hold accordingly.

\end{thm}

\begin{proof} We only need to show \textquotedblleft{}$\Longrightarrow$\textquotedblright{} of (i). For this let $P\in(\LBr,\Cregsur)$ be projective. We may write $P$ as a quotient of a countable direct sum of $\ell^1(\Gamma)$s and obtain a $\Cregsur$-conflation
$$
\ker S\longrightarrow\Bigosum{n\in\NN}{}\ell^1(\Gamma_n)\stackrel{S}{\longrightarrow}P
$$
which then splits. It follows that $P$ is isomorphic to a complemented subspace of $\bigoplus_{n\in\NN}\ell^1(\Gamma_n)$. As in part \textcircled{4} of the proof of Theorem \ref{PROP-PROJ}, it follows from \cite{DomProj92} that $P$ is isomorphic to a direct sum of $\ell^1(\Gamma)$s itself.
\end{proof}

\begin{rmk}In the categories $(\LBc,\CC)$ with $\CC$ being equal to $\Dcom$, $\Emaxcom$ or $\Etopcom$, and in $(\LBr,\EE)$ with $\EE\in\{\Emaxreg,\Etopreg\}$, any object that is isomorphic to a countable direct sum of $\ell^1(\Gamma)$s is projective. It is unknown if these are all projectives and if these categories have enough projectives.\diam{}
\end{rmk}

We end this paper by noting the following consequence, that extends a classic result for abelian categories with enough projectives to the categories of all/regular LB-spaces when endowed with their maximal deflation-exact structures, cf.~\cite[Props 3.30 and 3.31]{HR19}.

\begin{cor}\label{Kproj-cor} Let $*\in\{-,\text{b}\}$. There are natural triangle equivalences:\vspace{3pt}
\begin{compactitem}
\item[(i)] $\K^{\boldsymbol{\ast}}(\Proj(\LB,\Dmax))\stackrel{\hspace{-2pt}\sim}{\longrightarrow}\DD^{\boldsymbol{\ast}}(\LB,\Dmax)$,\vspace{3pt}
\item[(ii)]$\K^{\boldsymbol{\ast}}(\Proj(\LBr,\Cregsur))\stackrel{\hspace{-2pt}\sim}{\longrightarrow} \DD^{\boldsymbol{\ast}}(\LBr,\Cregsur)$.
\end{compactitem}
\end{cor}
\begin{proof} Let $(\mathcal{A},\CC)$ be either $(\LB,\Dmax)$ or $(\LBr,\Cregsur)$ and note that for any $A\in\mathcal{A}$ we find a resolution as in \eqref{B-RES} with $\mathcal{B}=\Proj(\mathcal{A},\CC)$ and $n=1$ by Theorems \ref{PROP-PROJ} and \ref{PROP-PROJ-2}. The proof of Theorem \ref{MAIN-1} then shows that the second functor below is an equivalence
$$
\K^{\boldsymbol{\ast}}(\Proj(\mathcal{A},\CC))\rightarrow\DD^{\boldsymbol{\ast}}(\Proj(\mathcal{A},\CC);\mathcal{A},\CC)\rightarrow \DD^{\boldsymbol{\ast}}(\mathcal{A},\CC).
$$
The first functor is by construction essentially surjective and it is fully faithful by \cite[Prop 3.28]{HR19}.
\end{proof}

\vspace{-5pt}

\begin{center}
{\sc Acknowledgments}
\end{center}

{

\small 
The author would like to thank Jos\'{e} Bonet and Bernard Dierolf for several helpful discussions, some of which dating more than 10 years back. Further he would like to thank Adam-Christiaan van Roosmalen for answering many questions about the content of \cite{HR19} and for his help with Proposition \ref{VRT-UNF}. Finally, he would like to thank Jochen Wengenroth for his help with Lemma \ref{LEM-W}.

}


\end{document}